\documentclass[10pt]{article}
\usepackage[utf8]{inputenc}

\usepackage{amsfonts}
\usepackage{placeins}
\usepackage{amsthm}
\usepackage{amssymb,color}
\usepackage{mdframed}
\usepackage{subcaption} 
\usepackage{cancel}
\usepackage[leqno]{amsmath}
\usepackage[unicode]{hyperref}
\hypersetup{colorlinks, linkcolor={blue},citecolor={green}, urlcolor={blue}}

\theoremstyle{plain}
\newtheorem{theorem}{Theorem}[section]
\newtheorem{lemma}[theorem]{Lemma}
\newtheorem{proposition}[theorem]{Proposition}

\newtheorem{assumption}{Assumption}
\theoremstyle{definition}
\newtheorem{definition}[theorem]{Definition}
\newtheorem{remark}[theorem]{Remark}

\usepackage[top=2.4cm,bottom=2.4cm,left=1.8cm,right=1.8cm,headsep=10pt]{geometry}

\usepackage[shortlabels]{enumitem}
\usepackage{nameref}
\usepackage{amsmath}
\usepackage{array}
\usepackage{url}

\makeatletter

\usepackage{todonotes}
\DeclareMathOperator*{\minimize}{minimize}

\def \E{\mathbb{E}}
\def \F{\mathbb{F}}

\def \N{\mathbb{N}}

\def \P{\mathbb{P}}

\def \R{\mathbb{R}}

\def\Ac{{\cal A}}
\def\Bc{{\cal B}}
\def\Cc{{\cal C}}

\def\Fc{{\cal F}}

\def\Lc{{\cal L}}

\def\Pc{{\cal P}}

\def\Tc{{\cal T}}
\def\Uc{{\cal U}}

\newcommand{\nico}[1]{\textcolor{black}{#1}}

\newcommand{\green}[1]{\textcolor{black}{#1}}

\newcommand{\val}{\text{val}}

\title{Investment and Operational Planning for electricity markets with massive entry of renewable energy}
\author{Nicolás HERN\'ANDEZ-SANTIB\'A\~NEZ\footnote{Departamento de Matemática, Universidad Técnica Federico Santa María, Chile. nicolas.hernandezs@usm.cl. 
}, Alejandro JOFR\'E\footnote{Departamento Ingeniería Matemática and Centro de Modelamiento Matemático (CNRS IRL2807), Universidad de Chile, Santiago, Chile. ajofre@uchile.cl}, Matías VERA\footnote{ETH Zürich, Department of Mathematics, Switzerland. matias.veravillalobos@math.ethz.ch
}
}
\date{\today}

\begin{document}
\setlength{\parindent}{0cm}
\setlength{\parskip}{0.2cm}
\maketitle

\begin{abstract}
In wholesale electricity markets, electricity producers and the \emph{independent system operator} (ISO) play a central role. The ISO is responsible for minimizing production costs while satisfying supply–demand balance and capacity constraints. In this paper, we study a continuous-time problem in which the ISO seeks to minimize the joint cost of operation and investment in an electricity network. The problem is formulated in terms of operational and investment control variables. We analyze the hierarchy between these controls and use the so-called \textit{Day-Ahead Problem} to find an explicit form of the optimal operation. This allows us to reformulate the \nico{investment} problem as a stochastic control problem with state constraints. We extend the results of state-constrained stochastic control to fit our setting. In particular, we use a version of the \textit{Pointing Inward Condition} to fully characterize the value of the problem as the unique viscosity solution of a constrained HJB equation. We then assign a specific dynamic to the capacity-demand process and discuss how the assumptions for the  HJB characterization result in a budget constraint for the planning. Finally, we run simulations for a three-node setting that resembles the Chilean market. We analyze short-, medium-, and long-term planning scenarios and discuss how to transition toward a system with high penetration of renewable energy.

\vspace{5mm}

\vspace{5mm}


\end{abstract}

\paragraph*{Acknowledgments.} This work was supported by Centro de Modelamiento Matemático (CMM) BASAL fund FB210005 for center of excellence from ANID-Chile, and ANID FONDECYT Iniciación 11240944 from ANID  (Chile), Basal project of ANID FB210005 (Chile), and the Consolidation Grant 2022 project COG224 from Seri (Switzerland). The third author gratefully acknowledges partial support by the SNF project MINT 205121-21981 (Switzerland).


\section{Introduction}\label{sec:Intro}

\nico{Energy is one of the most important components in the development of modern society.} However, pollution and climate change have become significant threats to living standards and future generations. Renewable energy appears to be one of the possible solutions to this problem. Furthermore, many countries are pursuing ambitious renewable targets; in particular, the Chilean energy industry aims to become completely renewable in the near future. In the \textit{Climate Change Conference} (COP25) Chile committed to producing 70\% of its total energy from clean and renewable sources by 2030.\footnote{For more information, see \url{https://cop25.mma.gob.cl/legadocop25/}} The transition to cleaner energy has several advantages and challenges associated with it; for instance, renewable energy is usually cheaper, and countries like Chile have the territorial conditions required for massive production; however, the capacity of renewable energy sources is strongly subject to uncertainty, complicating long-term planning. \nico{In this paper, we examine the complexity arising from the massive entry of renewable energy into an existing electricity market under uncertainty, primarily induced by the stochastic nature of renewable generation. We propose a continuous-time model that addresses both investment and operational decisions within the electricity network, formulated as a stochastic control problem.} Compared to discrete-time stochastic optimization approaches, such as those surveyed by Anderson et al. \cite{andersonetal2025} and Roald et al. \cite{roald2023}, our framework provides complementary insights into the optimal strategies. More broadly, there is a growing need for sector-specific, policy-oriented dynamic models that incorporate uncertainty and technological change (see Barbrook-Johnson et al. \cite{barbrook2024}). The class of stochastic optimal control models developed in this paper provides a natural and flexible framework to address these challenges.



\nico{The proposed model incorporates some particular features of the electricity market. In particular, it is a non-local market, meaning that geographical factors must be taken into account when modeling the entire transmission system. This spatial dimension is typically represented in the network through transmission constraints with thermal losses and has been extensively studied in static models.} Escobar and Jofre \cite{Jofre2010Monop} studied a pool type electricity market with declaration of prices and energy allocation by an independent system operator (ISO), who considers the transmission constraints under quadratic resistance losses. Similar constraints are considered in works as Aussel, Correa, and Marechal \cite{aussel2013electricity}, Aussel, {\v{C}}ervinka, and Marechal \cite{aussel2016deregulated}, and Henrion, Outrata, and Surowiec \cite{henrion2012analysis}, mostly interested in electricity auctions. In general, the quadratic resistance is well accepted when working in static problems with power transmission however, when considering a continuous-time setting, there are not many considerations of it. Aid, Campi, Huu, and Touzi \cite{aid2009structural} consider the demand process in a large geographical region and see it later as residual demand, avoiding thus the network structure inherent in the electric market and, as a consequence, the transmission constraints. Similarly, Aid, Campi, Langrené, and Pham \cite{aid2014probabilistic} consider a novel switching dynamic for the capacities of the producers, while ignoring the network structure of the market. In a recent work, Hernández-Santibáñez, Jofré, and Possamaï \cite{hernandez2023pollution} proposed a model with a network structure and quadratic losses, which studies how remunerations or fines, provided by the ISO in the form of contracts, may affect pollution levels over time. In contrast to the static case studies, in their model the ISO is the upper--level player and the generators are the lower--level players. Our model complements that of \cite{hernandez2023pollution} by incorporating new aspects that are inherent to renewable energy. In our formulation, capacity is not fixed, as we consider it to be a controlled stochastic process. \nico{Therefore, producers may not be able to satisfy their local demand and transmission in the network plays a major role.} By resorting to the static version of the planning problem, which has been well--studied, we are able to transform the problem of interest into a state--constrained stochastic control problem.


\nico{Most investment and capacity–expansion problems in electricity markets are typically formulated as mixed--integer stochastic programming models, often in a static or discrete multi–period setting. For an extensive review of these models, we recommend the book by Conejo, Carrión, and Morales \cite{conejo2010decision} or the article by Sagastizábal \cite{sagastizabal2012divide} as well as the references therein. These approaches are well suited for policy evaluation and market design, but they rely on finite--dimensional decision variables, see for instance \cite[Section 5]{borges2023distributionally} or \cite{fisher2008optimal, khodaei2010transmission}. In contrast, continuous--time stochastic control methods have been mainly applied to problems such as optimal operation, storage management, and real options analysis, while their use in capacity expansion and market design, particularly in the presence of network constraints, remains relatively limited. For this reason, the contribution of our paper is also methodological. We formulate the problem as a continuous–time stochastic control problem, in which investment and operational decisions evolve jointly, subject to network constraints. To tackle such problem, we reformulate it as a state--constrained problem. }


    The study of state--constrained control problems via the Hamilton-Jacobi-Bellman approach was initiated in the deterministic setting by Soner \cite{soner1986optimal,soner1986optimal2}, where the so-called \textit{Pointing Inward Condition} (PIC) 
    was introduced. Roughly speaking, this condition says that, on the boundary of the constraint-set, there is a choice of a control that leads the state to the interior of the set. Soner proved the value function to be the unique viscosity solution of a first order constrained Hamilton-Jacobi equation; this translates into being supersolution in the whole set while subsolution only in the interior. Still in a deterministic setting, Ishii and Koike \cite{ishii1996new} extended this result with their own PIC, proving that the value function is a subsolution on the boundary of a HJB equation, but for an inward pointing Hamiltonian that only considers the \textit{pointing inward controls}. In the stochastic setting, Lasri and Lions \cite{lasry1989nonlinear} studied a specific state-constrained problem, with only control of the drift and the identity matrix as volatility. The uncontrolled volatility forces one to consider an unbounded control set, which in the context of energy markets is usually undesirable when modeling the dynamics of a capacity process. Katsoulakis \cite{katsoulakis1994viscosity} studied the constrained problem with a compact control set, using the PIC and limiting the directions in which the process is allowed to move, so that it does not escape the set. The main assumptions in this work are that the constraint-set is $C^3$ and that the drift, volatility, and objective functions are bounded. This results in a characterization for the value function as the unique continuous viscosity solution of a second order constrained HJB equation. Ishii and Loreti \cite{ishii2002class} extended the PIC, asking in addition for the control to turn off the volatility. They characterized the value function as the unique viscosity solution of the second order constrained HJB and proved the subsolution property on the boundary of the set for the inward pointing Hamiltonian. In this work no regularity of the constraint-set is needed, nonetheless, there are technical limitations such as relying on the boundedness of the set, the drift, volatility, and objective function. Later, Bouchard and Nutz \cite{bouchard2012weak} extended the so-called weak dynamic programming principle, previously developed by Bouchard and Touzi \cite{bouchard2011weak}, to the case with state constraints. They dropped the boundedness assumptions on the set, the drift, the volatility and the objective function, and instead required local Lipschitz continuity and linear growth of the coefficients and the objective function. To obtain the HJB characterization of the value function, they assume that the lower semicontinuous envelope of the value function is of class $R(O)$, which is a variant of the PIC condition for sets with more general shapes. The authors provide an example in which this assumption holds in the case of a $C^1$ constraint set. Unfortunately, such a condition is not satisfied in our problem. \nico{Indeed, our constraint set is determined by the feasibility of the network operation, and it is non-smooth even in very simple examples. We therefore adapt the classical techniques in the literature, justifying that the PIC condition is an economically reasonable assumption in the electricity market, and proving that the value function of our problem is of class $R(O)$. This allows us to characterize the value function of our problem as the unique viscosity solution to the associated HJB equation.}


This paper is organized as follows. Section \ref{sec:CIP} formally introduces the model of the electricity network along with the planning problem. In Section \ref{sec:Coupling} we tackle the problem; we discuss the hierarchy of the controls, we introduce the so--called \emph{Day Ahead Problem}, and we use it to reformulate the initial problem as a stochastic control problem with state constraints. Section \ref{sec:State-Constr} is devoted to proving the HJB characterization of the value function; we adapt the results of Bouchard and Nutz \cite{bouchard2012weak} to a Lagrange formulation, and we prove that the extension of the PIC allows to obtain the comparison principle needed for the characterization. Section \ref{sec:ito-process} assigns a particular dynamic to the capacity-demand process, and we discuss that under this dynamic, the assumptions from the previous chapter reduce to a budget constraint. We then prove that under the correct budget, the value function is uniquely characterized by the HJB equation. Finally, in Section \ref{sec:Num-Results} we provide numerical results for networks with three nodes.

\textbf{Notations:} We let $\N := \{ 0,1,2,\dots\}$ be the set of natural numbers $\N^{\star} := \N\setminus\{ 0\}$. For $l \in \N^\star$, we denote by $\R^l_+$ the set of $l$-dimensional elements with real and positive coordinates. For $x \in \R^l$, $x^i$ denotes the $i$-th coordinate of the vector $x$. For $\gamma >0$, $B(x,\gamma)$ denotes the Euclidean ball centered in $x$ with radius $\gamma$. For $x,y,z\in \R^l$, we say that $x \in [y,z]$ if and only if $x^i \in [y^i,z^i]$ for each $i \in \{ 1,\dots,l\}$. For $d,k\in \N^\star$ we denote by $\R^{d,k}$ the set of matrices with real entries, $d$ rows, and $k$ columns. $\cdot^\top$ denotes the transpose operation in $\R^{d,k}$. We denote $S^d\subseteq \R^{d,d}$, the set of symmetrical matrices.  For a function $\varphi(t,x)$ we denote by $\partial_t\varphi$ its time partial derivative and $D\varphi$ its spatial gradient. For $p,n\in\N^\star$ we denote by $C^{p,n}([0,T] \times \R^l,\R^k)$ the space of all functions $f : [0,T]\times \R^l \to \R^k$ that are $p$ times continuously differentiable in the first variable, $n$ times continuously differentiable in the second variable. For any $C \subseteq \R^l$ we denote $\Bc(C)$ the Borel $\sigma$-algebra with respect to the trace topology. We denote by $\overline{C}$, $C^\circ$ and $\partial C$ the closure, interior, and boundary of $C$, respectively. For the probability spaces $(X_1,\Cc_1,\mu_1)$ and $(X_2,\Cc_2,\mu_2)$, we denote $\mu_1 \otimes \mu_2$ as the product measure on $X_1 \times X_2$. For $T>0$ we denote $\lambda$ as the normalized Lebesgue measure on $[0,T]$. 

\section{The planning problem}\label{sec:CIP}

We model an investment and operational planning problem in which an entity, called from now on the \textit{independent system operator} (ISO), dictates the continuous-time operation of a network with transmission losses by taking into account the production and investment costs. Each generator has a capacity process and a demand process. Since the generators are fueled by both renewable and non-renewable energy, they have uncertainty associated with the joint process. We model the uncertainty of the maximum capacity as an Itô process and the planning of the demand as a deterministic process.

\textbf{Network:} To represent the geographical component of the electric market, we consider a network structure for the problem. We take $(V,E)$ a graph where each node $i \in V$ represents a producer who has an associated capacity and demand $(Q^i,D^i)$ and each edge $e \in E$ represents a transmission line with flow limits $(\underline{\phi}^e, \overline{\phi}^e)$ and resistance $r^e \geq 0$. To simplify the notation, we consider without loss of generality that $V = \{ 1 , \dots, N\}$, where $N \in \N^{\star}$ is the number of locations and $E = \{e_1,\dots,e_M \}$ where $M \in \N^\star$ is the number of transmission lines. For each $i \in \{1 ,\dots, N\}$, we denote $K_i$ as the edges connected to the node $i$ and for each $e \in K_i$, $sgn(e,i)$ is either $+1$ or $-1$ to represent a fixed direction arbitrarily chosen.\footnote{Flow can be sent in both directions, $\phi^e >0$ represents the default direction and $\phi^e<0$ the opposite one.} 

\medskip
An interesting example is the three node case, which can be related to the Chilean geography, where the nodes represent \textit{North, Center} and \textit{South} , as seen in Figure \ref{graph:Chile}.
\begin{figure}[h]
\begin{center}
    \begin{tikzpicture}[node distance=2cm]

        \node[draw, circle, align=center] (node0) at (0,3) {
            1 };

        \node[draw, circle, align=center] (node1) at (0,0) {
            2};

        \node[draw, circle, align=center] (node2) at (3,1.5) {
            3 };

        \draw[->, thick] (node0.south) -- (node1.north)
            node[midway, left, align=center] { $r^{e_1}, \underline{\phi}^{e_1}, \overline{\phi}^{e_1}$ };

        \draw[->, thick] (node0.east) -- (node2.north west)
            node[midway, above right, align=center] {$r^{e_2}, \underline{\phi}^{e_2}, \overline{\phi}^{e_2}$};

        \draw[->, thick] (node2.south west) -- (node1.east)
            node[midway, below right, align=center] {$r^{e_3}, \underline{\phi}^{e_3}, \overline{\phi}^{e_3}$};
    \end{tikzpicture}
    \caption{Three node representation of Chile}
    \label{graph:Chile}
\end{center}
\end{figure}
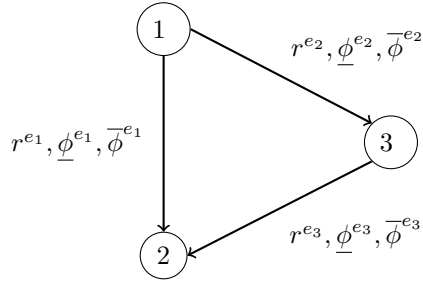
The general network structure just described has been studied in the static case, see for instance \cite{aussel2016deregulated,aussel2013electricity,Jofre2010Monop}, and in a continuous--time setting in \cite{hernandez2023pollution}.

\textbf{Demand and capacity process:} As stated above, each node $i$ has an associated capacity and demand process. In general, we consider a joint Itô process. We do not limit the results to a specific type of dynamic, nonetheless in our later examples and simulations we consider the uncertainty only in the capacity and we assume that the demand follows a deterministic planification.

Let $(\Omega, \Fc, \P,\F)$ be a filtered probability space and $T>0$. Given a compact set $U \subseteq \R^k$, with $k \in \N^\star$, let $\Uc$ be the set of $U$-valued predictable processes $\nu$, these will be called controls. Now we assume that for each control $\nu \in \Uc$ one has a pair of $\F$-predictable processes $(Q^\nu_s, D^\nu_s)$. Here for $(t,\omega) \in [0,T] \times \Omega$, $Q^\nu_t(\omega)\in\R^N_+$ represents the effective maximum capacity available at each node \green{and time $t$, whereas} $D^\nu_t (\omega)\in\R^N_+$ represent the demands at each node \green{and time $t$}. 

\textbf{Network operation:} Maintaining the functionality of the network requires three components: the production must be in accordance to the capacity at each node, the flows must respect the flow limits, and the production and the flows at each node must fulfill the associated demand. Let us consider $\Ac$ the set of $\F$-predictable processes $(q,\phi)$ taking values on $\R^N \times \R^M$. \nico{For a given process $(Q_s,D_s)_{s \in [0,T]}$,} the network operation is reduced to the following constraints
\begin{align*}
    q^i_s + \sum_{e \in K_i} sgn(e,i)\phi^e_s \geq D^i_s + \sum_{e \in K_i} \frac{r^e}{2}(\phi^e_s)^2, ~ i=1,\dots,N&, \quad \forall s \in [0,T],~ \P\text{-a.s.} \\
    q_s^i \in [0,Q^i_s], \phi_s^e \in [\underline{\phi}^e,\overline{\phi}^e], ~ i=1,\dots,N,\forall e \in E&, \quad \forall s \in [0,T],~  \P\text{-a.s.}
\end{align*}
The first constraint is essentially that the amount produced at the node $q^i$ plus (or minus) the energy from the flows connected to the node must be greater than the demand and the transmission losses that are shared by both connected nodes\footnote{We incorporate quadratic losses in the model as justified in \cite{Jofre2010Monop}.}. The second constraint bounds the production to the capacity, and the flows to the flow limits. The reason behind \green{the sample-path constraint} is that the electric market considers as a ``critical case'' one in which the demand is not fulfilled. Thus, we are interested in  avoiding such critical scenarios.

\begin{remark}
Alternatively, one may relax the first almost-sure constraint by replacing it with a chance constraint, requiring instead that the probability of meeting demand be at least some level $p \in [0,1]$. However, in the context of electricity markets this interpretation corresponds to allowing demand rationing with positive probability, which is typically unacceptable for producers and regulators. For this reason, we restrict our analysis to the case $p = 1$.
\end{remark}

\textbf{Costs.} \nico{We distinguish two types of costs considered by the ISO: production costs and investment costs at each node. In practice, these two components are measured on different scales. To make them comparable, we express a continuation of the investment costs, on an annual or monthly basis, which allows us to balance the magnitudes of both types of costs.}

We consider arbitrary production cost functions $c^i:  \R_+ \times \R_+ \to \R_+$ under suitable conditions specified in Assumption \ref{assump:1}. As mentioned below, some of the conditions are technical while others come from the modeling.\footnote{ The main idea of the production cost is to represent two different unit costs for renewable $(c_R)$ and nonrenewable $(c_{NR})$ technologies. For instance, an interesting form of the costs function is given by
\begin{align*}
    c^i(q,Q) = \int_0^{q^i} c^i_R \cdot 1_{[0,Q^i-Q^i_{NR}] } (s) +  c^i_{NR} \cdot 1_{(Q^i - Q^i_{NR},\infty)}(s) \mathrm{d}s,
\end{align*}
with $0<c_R^i\leq c_{NR}^i$ and $Q^i_{NR}\geq 0$. In this context, the cheaper source of energy is always preferred.} The investment cost functions $h^i:  U\times \R_+ \times \R_+ \to \R_+$ have different interpretations depending on how one controls the capacity-demand process. In the general setting, we do not make major assumptions in them
. In the model presented in Section \ref{sec:ito-process}, we separate the investment variable into $\nu = (\mu,\alpha)$, where $\mu$ will be the investment in the acquisition of renewable technologies, for instance solar panels, while $\alpha$ will be the stabilization factor associated to the capacity process. We will also separate the costs $h^i = h^i_\mu + h^i_{\alpha}$, with $h^i_\mu$ the cost of technology and installation, and $h^i_\alpha$ the unitary cost of stabilizing the volatility. 


\textbf{Planning problem:} With all the components needed to model the problem, the ISO's goal is to keep the network operating correctly throughout the period while minimizing the total cost of operation and investment. Since the problem is set over $[0,T]$ and there is uncertainty in the capacities, the resulting problem of the ISO is described by
\begin{align*}\label{prob:generalprob}\tag{$P$}
    \minimize_{(\nu,q,\phi)\in \Uc \times \Ac} \hspace{.3cm}& \E \left[  \int_0^T \sum_{i=1}^N c^i(q^i_s,Q^{i,\nu}_s) + h^i(\nu_s,Q^{i,\nu}_s, D^{i,\nu}_s) \text{d}s\right] \\
    \text{s.t.} ~~~~~~& q^i_s + \sum_{e \in K_i} sgn(e,i)\phi^e_s \geq D^{i,\nu}_s + \sum_{e \in K_i} \frac{r^e}{2}(\phi^e_s)^2, ~ i=1,\dots,N, ~ \forall s \in [0,T],~ \P\text{-a.s.} \\
    &q_s \in [0,Q^{\nu}_s], \phi_s \in [\underline{\phi},\overline{\phi}], ~  \forall s \in [0,T],~  \P\text{-a.s.},
\end{align*}
where we use the standard convention $\inf\{\emptyset\} = +\infty$. 

\begin{remark}\label{remark:centralized-problem}
Our model considers a centralized market in which the ISO controls both operational and investment decisions and seeks to minimize total cost. The ISO’s investment decisions can also be interpreted as recommendations to producers, provided that it has complete information about their cost functions and aims to minimize social cost. The structure of the problem changes significantly when investment decisions are made independently by producers. At the end of Section \ref{sec:conclusion}, we discuss how the problem can be reformulated in such setting, how some of our results may extend, and the main challenges that arise.
\end{remark}

\section{Reformulation of the control problem}\label{sec:Coupling}

The main idea to tackle problem \eqref{prob:generalprob} is to understand the hierarchy over the decision variables. Once the investment control is placed, the production plan control may be decided directly from the capacity-demand process. This means that the control $\nu \in \Uc$ has a greater importance in regard to the value of the minimization problem, or at least, has more freedom for choosing its values. It takes only into consideration that there is room to choose $(q,\phi)\in\Ac$ to maintain the network operation functioning.

\subsection{Hierarchy of controls} 
Let us translate the idea of hierarchy in the decision variables into a mathematical expression. For this we need to define what it means to have a feasible pair $(Q,D)$.
\begin{definition}
    We say that a pair $(Q,D) \in \R^N_+ \times \R^N_+$ is \textit{feasible} if
    \begin{align*}
        \exists q \in [0,Q], \exists \phi \in [\underline{\phi}, \overline{\phi}]\text{ such that } q^i + \sum_{e \in K_i} sgn(e,i)\phi^e \geq D^i + \sum_{e \in K_i} \frac{r^e}{2}(\phi^e)^2,~i=1,\dots,N.
    \end{align*}
    We say that $G: \R^N_+ \times \R^N_+ \to \R^m$, with $m\in \N^\star$, is a \emph{feasibility function} if for any $(Q,D) \in \R^N_+ \times \R^N_+$ one has that $G(Q,D) \geq 0$ implies that $(Q,D)$ is a feasible pair.
\end{definition}

We will use a feasibility function to interchange the role of $(q,\phi)$, in the correct operation of the network, and give it to $(Q^\nu,D^\nu)$. To solve the problem \eqref{prob:generalprob}, we will ignore the network structure when choosing $\nu$ and focus only in the need to maintain the network operation. Therefore we will use a feasibility function that plays the role of sufficient and necessary condition for the network operation (see Remark \ref{rem:feasibiliy-function}). Let us define thus $G : \R^N_+ \times \R^N_+ \to \R$ as follows
\begin{align}\label{func:G}
    G(Q,D) := \max_{(q,\phi) \in [0,Q] \times [\underline{\phi}, \overline{\phi}]} \min_{i=1,\dots, N}  q^i - D^i + \sum_{e \in K_i}\left( sgn(e,i)\phi^e - \frac{r^e}{2}(\phi^e)^2 \right).
\end{align}
It is quite obvious that this is a feasibility function and also $(Q,D)$ being a feasible pair implies that $G(Q,D)\geq 0$. From here on every time we denote $G$ we refer to \eqref{func:G}.

\begin{remark}\label{rem:feasibiliy-function}
The feasibility function has a role of a necessary but not sufficient condition for the network operation. This `relaxed' concept allows us to open the model to different structures needed over the network \green{as detailed in Section \ref{sec:Intro}}. For example, if each generator wants to have enough capacity to satisfy strictly more than its demand, then one may consider a slackness $\alpha \in\R^N_+$ and the feasibility function $G(Q,D) = (Q^i - D^i - \alpha^i)_{i=1,\dots,N}$.     
\end{remark}

Returning to the control variables $(q,\phi)$, we state that in the hierarchy of controls, these are the last ones to be chosen. The general meaning of this, is that one can consider a sub-problem once $\nu$ is already chosen, subject to $(Q^\nu,D^\nu)$ being a feasible pair (at every instance in almost surely every trajectory). To model such sub-problem, let us define the following set
\begin{align*}
    \Uc_G := \{ \nu \in \Uc  : G(Q^\nu_s,D^\nu_s) \geq 0, \forall s\in [0,T],~\P\text{-a.s.} \}.
\end{align*}
If $\Uc_G = \emptyset$, then \eqref{prob:generalprob} has an empty feasible set so its value is directly $+\infty$. We may now consider the sub-problem of minimizing the production costs under the network operation constraints. For each $\nu \in \Uc_G$ we consider
\begin{align}\label{prob:nu-subproblem}\tag{$P_{\nu}$}
    \minimize_{(q,\phi)\in \Ac} ~& \E \left[  \int_0 \sum_{i=1}^N c^i(q^i_s,Q^{i,\nu}_s) \text{d}s\right] \\
    \text{s.t.} ~& q^i_s + \sum_{e \in K_i} sgn(e,i)\phi^e_s \geq D^{i,\nu}_s + \sum_{e \in K_i} \frac{r^e}{2}(\phi^e_s)^2, ~ i=1,\dots,N ~ \forall s \in [0,T],~ \P\text{-a.s.} \label{const:supply>demand} \\
    &q_s \in [0,Q^{\nu,\mu}_s], \phi_s \in [\underline{\phi},\overline{\phi}], ~  \forall s \in [0,T],~  \P\text{-a.s.} \label{const:capacity-const}
\end{align}
This problem is a continuous and stochastic form of a (simplified) problem known as the \textit{Day ahead problem} (DAP), which will be properly defined in subsection \ref{sec:DAP}. We now notice that the hierarchy of decision is explicit by using \eqref{prob:nu-subproblem}. Indeed, let us define $\Ac_\nu$ as the processes in $\Ac$ that satisfy \green{Constraints \eqref{const:supply>demand} and \eqref{const:capacity-const}}. By defining $V = \val(P)$ and $V_\nu = \val(P_\nu)$, one obtains that
\begin{align*}
    V 
    = \inf_{\nu \in \Uc_G} \left\{ V_\nu + \E\left[ \int_0^T \sum_{i=1}^N h^i(\nu_s, Q^{i,\nu}_s, D^{i,\nu}_s) \text{d}s  \right] \right\}.
\end{align*}
This last expression removes completely the dependence of $V$ on the controls $(q,\phi)$, interchanging them with a state constraint condition and the value function $V_\nu$. As said before, the problem \eqref{prob:nu-subproblem} resembles a continuous and stochastic version of the (DAP) so it is rather natural to think that the optimal process $(q, \phi)$ is not far from the optimal points of the (DAP). This intuition is the key that will allow us to find an explicit expression to $V_\nu$ and transform the original problem into a state constrained problem. In the next section, we focus on understanding the standard (DAP) problem.

\subsection{The day ahead problem}\label{sec:DAP}

The (DAP) is the lower problem of an auction in which each generator (or node in our network structure) declares his prices, capacities and demands, then the ISO chooses the production at each node along with the flows at each edge so that the demand constraints are satisfied and the production cost is as low as possible. The upper problem in this auction is that each generator wants to maximize the payment obtained from the ISO. This auction-type game happens repeatedly and the time between two repetitions is short, usually no more than a day, therefore its name. Given the sets of cost functions, \nico{demands and capacities $c^i :\R_+ \times \R_+ \to \R_+$, $D^i\in \R_+$ and $Q^i \in \R_+$, respectively for $i=1,\dots,N$,} we define
\begin{align*}\label{prob:DAPQD}\tag{$DAP$}
    \minimize_{(q,\phi) \in \R^N \times \R^M} ~ & \sum_{i=1}^N c^i(q^i, Q^i) \\ 
    \text{s.t.}~~~~~~& q^i + \sum_{e \in K_i} sgn(e,i)\phi^e \geq D^i + \sum_{e \in K_i} \frac{r^e}{2}(\phi^e)^2,~i=1,\dots,N, \\
    & q \in [0,Q], \phi \in [\underline{\phi},\overline{\phi}].
\end{align*}
This problem under suitable conditions is well behaved. To simplify further the notation we introduce the functions
\begin{align*}
    T^i (q,\phi;D) = q^i - D^i + \sum_{e\in K_i} \bigg( sgn(e,i)(\phi^e) - \frac{r^e}{2}(\phi^e)^2\bigg), \quad i=1,\dots,N.
\end{align*}
\begin{assumption}[Network structure]\label{assump:1} ~ \\
 $(i)$ For each $i=1,\dots,N$, $c^i$ is continuous, convex, strictly increasing in the first variable and measurable in the second variable. Moreover, $c^i(0,\cdot) \equiv 0$ and $\lim_{q \to 0^+} c^i(q,\cdot)/q>0$. \\
 $(ii)$ For each $e \in E$, $r^e > 0$.
\end{assumption}
As mentioned in the previous section, Assumption \ref{assump:1} is important for both technical and modeling reasons. Aiming to represent a realistic geographically differentiated electric market, we assume in $(i)$ that not producing has no cost and there is no unit of energy for free, while in $(ii)$ we assume that there is always a loss in the transmissions.
\begin{proposition}\label{prop:KirchhoffLaw}(Kirchhoff Law) Let Assumption \ref{assump:1} hold true and additionally assume that \eqref{prob:DAPQD} is feasible. Then for any solution $(q,\phi)$ of \eqref{prob:DAPQD} one has that $T^i(q,\phi;D) = 0, \forall i = 1,\dots,N.$
\end{proposition}
\begin{proof}
    The proof follows by the same arguments as in the proof of \cite[Lemma 3.1]{aussel2016deregulated}. Notice that the \nico{monotonicity} assumption on $c$ is enough to get the same contradiction.
\end{proof}
\begin{proposition}\label{prop:DAPuniqueness}
    Let Assumption \ref{assump:1} hold and additionally assume that \eqref{prob:DAPQD} is feasible. Then there exists a unique solution of \eqref{prob:DAPQD}.
\end{proposition}
\begin{proof}
    Notice that the feasible set is compact, which gives us the existence of optimal solutions. Let $(q_0,\phi_0),(q_1,\phi_1)$ be two solutions of \eqref{prob:DAPQD} and let us assume that $\phi_0 \not = \phi_1$. Then there exists $i \in \{ 1,\dots,N\}$ such that $(\phi_0^e)_{e\in {K_i}} \not = (\phi_1^e)_{e\in {K_i}}$\footnote{ $(\phi^e)_{e \in K_i}$ denotes the projection of $\phi^e$ in which we only consider the coordinates $e \in K_i$}. For $(q_\lambda,\phi_\lambda) := \lambda(q_1,\phi_1) + (1-\lambda)(q_0,\phi_0)$ one has that, due to the strict concavity of $T^i(q,\phi;D)$ with respect to $(\phi^e)_{e\in K_i}$, 
    \begin{align*}
        T^i(q_\lambda,\phi_\lambda;D) > \lambda T^i(q_1,\phi_1;D) + (1-\lambda)T^i(q_0,\phi_0;D) \geq 0.
    \end{align*}
    However, since \eqref{prob:DAPQD} is a convex problem, we have that $(q_\lambda,\phi_\lambda)$ is a feasible pair for each $\lambda \in [0,1]$, thus contradicting Proposition \ref{prop:KirchhoffLaw} obtaining that $\phi_0= \phi_1$. Finally $q_0=q_1$ follows directly from the Kirchhoff's equality and the uniqueness of the optimal flows.
\end{proof}
\begin{proposition}\label{prop:Kcloseandconvex}
    $K := G^{-1}(\R_+)$ is convex and closed.
\end{proposition}
\begin{proof}
    The convexity of the set $K$ can easily be deduced from \eqref{func:G}. From Berge's Maximum Theorem, (see \cite[Theorem 17.31]{Charalambo2006}) we obtain that $G$ is continuous. This implies that $K$ is closed.
\end{proof}
\begin{definition}
    For each $(Q,D) \in K$, we define the optimal production pair $(q^{\star},\phi^{\star})(Q,D)$ as the unique solution to \eqref{prob:DAPQD}.
\end{definition}
\begin{proposition}\label{prop:measurability(q,phi)}
    Let Assumption \ref{assump:1} hold true. The function $(q^{\star},\phi^{\star})(\cdot, \cdot)$ is measurable on $(K,\mathcal{B}(K))$.
\end{proposition}
\begin{proof}
    We define the correspondence $\psi:\R^N\times\R^N \rightrightarrows \R^N\times\R^M$ by
\begin{align*} \psi (Q, D): =  \left\{ (q, \phi) \in [0,Q]\times [\underline{\phi}, \overline{\phi}]~ ;~ T^i(q,\phi,D)\geq 0   ~,~ i=1,\dots, N \right\},
    \end{align*} 
    which is clearly compact valued.  We will check now its upper hemicontinuity. Let us take a sequence $(Q_n,D_n) \to (Q,D)$ and $(q_n,\phi_n) \in \psi(Q_n,D_n),  \forall n\geq 0$. Since $(Q_n,D_n)$ is bounded, we have that $(q_n,\phi_n)$ is also bounded and it has a limit point $(q,\phi)$. By the continuity of all the functions $T^i$, we obtain that
    \begin{align*}
        q \in [0,Q], \phi \in [\underline{\phi},\overline{\phi}], \quad T^i(q,\phi,D) \geq 0, \quad \forall i=1,\dots, N.
    \end{align*}
    This means that $(q,\phi)\in \psi(Q,D)$ and therefore $\psi$ is upper hemicontinuous, which implies that the correspondence is measurable with respect to $\Bc(K)$. Finally, using the Measurable Maximum Theorem (see \cite[Theorem 18.19]{Charalambo2006}) we obtain the desired result.
\end{proof}

\subsection{Equivalent formulation of the problem}
For any $\nu \in \Uc_G$ our goal is to properly characterize the value function $V_\nu$. As said before, a simple idea is that if we fix a particular $\omega \in \Omega$ and $s\in [0,T]$, one faces a (DAP) problem, so the best choice would be to take $(q_s(\omega),\phi_s(\omega)) = (q^\star,\phi^\star)(Q^{\nu}_s(\omega),D^\nu_s(\omega))$. This reasoning is not entirely correct because we need to satisfy  $(Q^{\nu}_s(\omega),D^\nu_s(\omega))\in K$ to have $(q^\star,\phi^\star)$ well defined. Nonetheless, we know that with probability one this will be the case. We have then the following result.

\begin{proposition}\label{prop:LPuniqueness}
    Let Assumption \ref{assump:1} hold true and consider $\nu \in \Uc_G$. Then Problem \eqref{prob:nu-subproblem} has at least one solution. Moreover, a feasible pair $(\tilde{q},\tilde{\phi})$ is an optimal solution of \eqref{prob:nu-subproblem} if and only if
    \begin{align*}
        \lambda \otimes \P \left[ (\tilde{q},\tilde{\phi}) = (q^\star,\phi^\star)(Q^\nu,D^\nu) \right] = 1.
    \end{align*}
    Finally, one has that
    \begin{align*}
        V_\nu = \E \left[ \int_0^T \sum_{i=1}^N c^i (q^{\star,i}(Q^\nu_s ,D^\nu_s) , Q^{\nu,i}_s) \textnormal{d}s \right].
    \end{align*}
\end{proposition}
\begin{proof}
    Since $\nu \in \Uc_G$, define $\tilde{\Omega}:=\{ \omega\in\Omega: G(Q^\nu_s(\omega),D^\nu_s(\omega)) \geq 0, \forall s\in [0,T]\}$ such that $\P(\tilde\Omega)=1$. By using Proposition \ref{prop:Kcloseandconvex} and Proposition \ref{prop:measurability(q,phi)} one may extend $(q^\star,\phi^\star)(\cdot,\cdot)$, which is initially defined over $K$, to a measurable function in the complete space $\R^N\times\R^N$.\footnote{For instance, we define $(q^\star,\phi^\star)(Q,D)=0$ for every $(Q,D)\not\in K$.} We define then the process
    \begin{align*}
        (q^\star_s(\omega),\phi^\star_s (\omega)) := (q^\star,\phi^\star)(Q^{\nu}_s(\omega), D^\nu_s(\omega)),
    \end{align*}
    which satisfies $(q^\star,\phi^\star) \in \Ac$ due to the Borel mesurability. Moreover, $(q^\star,\phi^\star)$ is feasible for \eqref{prob:nu-subproblem} because the constraints are satisfied within $\tilde{\Omega}$. Let us now take an arbitrary $(q',\phi')\in\Ac_\nu$ and such that
        $\lambda \otimes \P [(q',\phi') \not = (q^\star,\phi^\star)]>0$.
    Define now $A = \left\{ (s,\omega)\in [0,T] \times \Omega  : (q^\star_s(\omega), \phi^\star_s(\omega) )\not = (q'_s(\omega),\phi'_s(\omega))\right\}$. We assume without loss of generality that\footnote{It is easy to check that the left-hand side set is equal to $\tilde\Omega$. We assume thus that the right-hand side set is $\tilde\Omega$.}
    \begin{align*}
     \tilde\Omega =  \left\{ \omega \in \Omega : T^i(q^\star_s(\omega),\phi^\star_s(\omega),D^\nu_s(\omega)) \geq 0, \forall i=1,\dots N\right\} =  \left\{ \omega \in \Omega : T^i(q'_s(\omega),\phi'_s(\omega),D^\nu_s(\omega)) \geq 0, \forall i=1,\dots N\right\}.
    \end{align*}
    If this is not the case, it is enough to do the rest of the proof in the intersection of both sets. Define $\hat{A} := A \cap (\tilde{\Omega} \times [0,T])$. Due to the uniqueness of minimizer proved in Proposition \ref{prop:DAPuniqueness} one has 
    \begin{align*}
        \sum_{i=1}^N c^i((q^\star)_s^i(\omega),Q^\nu_s(\omega))< \sum_{i =1}^N c^i (q'_s(\omega),Q^\nu_s(\omega)), ~\forall (s,\omega)\in\hat{A}.
    \end{align*}
    Now by taking expectation
    \begin{align*}
        \E_{\lambda\otimes\P} \left[ \sum_{i=1}^N c^i((q^\star)^i, Q^\nu) \right] 
        & < \E_{\lambda\otimes\P} \left[ \sum_{i=1}^N c^i((q')^i,Q^\nu) \cdot 1_{\hat{A}} \right] + \E_{\lambda\otimes\P} \left[ \sum_{i=1}^N c^i((q^\star)^i,Q^\nu) \cdot 1_{\hat{A}^c} \right]   = \E_{\lambda\otimes\P} \left[ \sum_{i=1}^N c^i((q')^i,Q^\nu)  \right],
    \end{align*}
and by Fubini's theorem (see \cite[Theorem 6.2.1]{san2018teoria}), one has that
\begin{align*}
    \E\left[ \int_0^T \sum_{i=1}^N c^i((q^\star)^i_s, Q^\nu_s)  \text{d}s\right] < \E\left[ \int_0^T \sum_{i=1}^N c^i((q')^i_s, Q^\nu_s)  \text{d}s\right].
\end{align*}
Conversely, if we consider a feasible $(q', \phi')$ such that $\lambda \otimes \P[(q^\star,\phi^\star) \not = (q',\phi')] = 0$, then the objective values of $(q', \phi')$ and $(q^\star,\phi^\star)$ are the same, therefore $(q', \phi')$ is a solution to \eqref{prob:nu-subproblem}.
\end{proof}

The main consequence of Proposition \ref{prop:LPuniqueness}, is that we can forget about the planning control $(q,\phi)$ involved in Problem \eqref{prob:generalprob} and we can work directly with the extension of the minimizing function $(q^\star,\phi^\star)$. We define therefore the constrained problem
\begin{align*}\tag{$\tilde{P}$}\label{prob:equivalent-problem}
    \minimize_{\nu \in \Uc} ~~~& \E \left[ \int_0^T \sum_{i=1}^N c^i(q^{\star,i}(Q^\nu_s,D^\nu_s), Q^\nu_s) + h^i(\nu_s, Q^{i,\nu}_s, D^{i,\nu}_s) \text{d}s \right], \\
                     \text{s.t.}~~~  & G(Q^{\nu}_s, D^\nu_s) \geq 0, ~ \forall s \in [0,T]  ~\P\text{-a.s.}
\end{align*}
which is equivalent to \eqref{prob:generalprob}. From now on, we will focus on solving Problem \eqref{prob:equivalent-problem}\nico{, which is a state-constrained stochastic control problem associated with the closed set $K$. To the best of our knowledge, there are no existing results in the literature that characterize its value function as the unique viscosity solution to the corresponding HJB equation in the case where $K$ is neither smooth nor bounded and the controls are bounded. Since this is precisely our setting, in the next section we make a brief detour from the ISO's problem and provide the desired characterization of the value function for problems satisfying a PIC condition, which admits a natural economic interpretation in the context of electricity markets (see Remark~\ref{remark:InwardPointing}).
 }

\section{Control problem with state constraints}\label{sec:State-Constr} 

Problem \eqref{prob:equivalent-problem} is a constrained stochastic control problem. This type of problem has been studied in many ways, commonly by asking for regularity on the domain and conditions over the running cost, drift and volatility in order to get a \textit{constrained} HJB equation. In \cite{soner1986optimal, soner1986optimal2} the constrained problem is studied in the deterministic setting and the \textit{pointing inward} assumption is introduced, which essentially means that at each boundary point there is a control that makes the drift point to the interior of the domain. The stochastic case has also been studied with similar results, see for instance \cite{ishii2002class,katsoulakis1994viscosity}.

There have also been more recent results over this problem in which the controllability assumptions are relaxed, see for instance \cite{bokanowski2010reachability} in the deterministic setting and \cite{bokanowski2016state} in the stochastic setting. In the stochastic case, which is our case, the relaxation in this matter comes at the price of solving an auxiliary problem instead, and obtaining an HJB equation different from the usual one obtained in constrained optimal control. Moreover, in this approach one needs an existence condition for an unconstrained control problem (see \cite[Condition H4]{bokanowski2016state}). Such condition is not bad by itself, but in our case it becomes rather difficult to handle since we do not have much information about the extension of $(q^\star,\phi^\star)$.

As of the knowledge of the authors, the closest results to our setting are the ones in \cite{bouchard2012weak}. Therefore, in this section we adapt those results in order to apply them to our problem. In order to limit the extension of this paper, we sketch some proofs in this section rather than completely writing them. \nico{We work in a more general setting than the one introduced in the previous sections, both to simplify the notation and to improve readability. In addition, we expect that this framework and some of our results could be useful for addressing other related problems in electricity markets.}

Let $\Omega = C([0,T],\R^N)$ be the space of functions with continuous paths, equipped with the Wiener measure $\P$, and let $W$ be the canonical process defined by $W_t(\omega) = \omega_t$. Denote by $\F = (\Fc_t)_{t \in [0,T]}$ the $\P$-augmentation of the filtration generated by $W$. \nico{We denote by $\Uc$ the set of $U$-valued predictable processes, where $U$ is compact and finite-dimensional.} For $t \in [0,T]$, we set $\Uc_t = \{ \nu \in \Uc : \nu \text{ is }\F^t\text{-predictable} \}$, where $\Fc^t = (\Fc^t_s)_{s\in [0,T]}$ is chosen to be the augmentation of $\sigma(W_r - W_t, t\leq r \leq s)$; due to the independence of increments of the Brownian motion, we directly have that $\Fc^t$ is independent of $\Fc_t$. We denote $\mathcal{T}^t$ the set of $\F^t$-stopping times with values in $[t,T]$. Let $b:\R^d\times U \to \R^d, \sigma : \R^d \times U\to \R^{d,N}$ be two Lipschitz continuous functions. For each $(t,x,\nu) \in [0,T] \times \R^d \times \Uc,$
we denote by $(X^{\nu,t,x}_s)_{s \in [0,T]}$ the strong solution of the SDE
\begin{align}\label{eq:SDE estandar}
    X_s = x + \int_t^s b(X_\tau, \nu_\tau) \text{d}\tau + \int_t^s \sigma(X_\tau, \nu_{\tau}) \cdot \text{d}W_\tau,
\end{align}
where we set $X^{\nu,t,x}_r = x$ for all $r\leq t$. Under these assumptions it is well known that the strong solution is unique, see for instance \cite[Theorem 1.3.15]{pham2009continuous}. Let now $f : \R^d \times U \to \R_+$ be a continuous function uniformly Lipschitz with respect to the space variable. We define the expected cost
\begin{align*}
    J(\nu;t,x) = \E\left[ \int_t^T f(X^{\nu,t,x}_s, \nu_s) \text{d}s \right].
\end{align*}
Next, given a set $C \subseteq \R^d$ with non-empty interior, such that $\overline{(C^\circ)} = C$, we consider the value functions
\begin{align*}
    V(t,x) = \inf_{\nu \in \Uc(t,x)} J(\nu;t,x) \quad,\quad V^\circ(t,x)  =  \inf_{\nu \in \Uc^\circ (t,x)} J(\nu;t,x),
\end{align*}
where the sets of controls are given by
\begin{align*}
    \Uc(t,x) := \{ \nu \in \Uc_t : X^{\nu,t,x}_s \in C, ~\forall s \in [0,T], \P\text{-a.s.} \},\quad 
    \Uc^\circ(t,x) := \{ \nu \in \Uc_t : X^{\nu,t,x}_s \in C^\circ, ~\forall s \in [0,T], \P\text{-a.s.} \}.
\end{align*}

Due to the randomization argument, (see \cite[Remark 5.2]{bouchard2011weak}), we have that working with $\Uc_t$ or $\Uc$ does not make a difference in the value of the problems just defined. It is direct that for each $(t,x) \in [0,T] \times C^\circ$ we have that $V(t,x) \leq V^\circ(t,x)$. Moreover, due to the continuity of the trajectories of $X$, for each $(t,x) \in [0,T]\times C^c$ we have that $V(t,x) = V^\circ(t,x) = \infty$ and for each $x \in C, y\in C^\circ$ we have the terminal conditions $V(T,x) = V^\circ(T,y) = 0$.

The reason for defining both value functions is that the HJB equation associated to both problems is the same. This is a hint that under reasonable conditions, both value functions will be equal to the unique viscosity solution of the same constrained HJB equation. To reach this conclusion, we will have to prove the super- and sub-solution properties as well as a comparison principle. Let us define what we refer to as a constrained equation.

\begin{definition}
    Given $F:[0,T] \times \R^d \times \R \times \R^d \times S^d \to \R$ and a closed set $\Pi \subseteq \R^d$, we say that $\varphi$ is a \textit{constrained solution} of the equation
\[
        F(t,x,\phi,D\phi,D^2\phi) = 0,
\]
on $[0,T)\times \Pi$, if it is a viscosity supersolution on $[0,T) \times \Pi$ and a viscosity subsolution on $[0,T)\times \Pi^\circ$.
\end{definition}

\subsection{Supersolution Property.}
The supersolution property is usually the simpler one when characterizing the value function as the unique viscosity solution of the HJB equation. Roughly speaking, we can prove the property for $V$ without worrying about the boundary of $C$.  

\begin{proposition}\label{prop:PPD-sup}
    Let $(t,x) \in [0,T] \times C$ and consider a family $\{ \tau^\nu, \nu \in \Uc(t,x)\} \subseteq \mathcal{T}^t$. Let $\phi:[0,T] \times C\rightarrow\R$ be a measurable function such that $V\geq \phi$. Then one has
    \begin{align*}
        V(t,x) \geq \inf_{\nu \in \Uc(t,x)} \E \left[ \int_t^{\tau^\nu} f(X^{\nu,t,x}_s, \nu_s)\mathrm{d}s + \phi(\tau^\nu,X^{\nu,t,x}_{\tau^\nu}) \right].
    \end{align*}
\end{proposition}
\begin{proof} It follows directly from the Lagrange form of the pseudo-Markov property, see \cite{claisse2016pseudo} and Lemma \ref{lemma:flow-property}; and using the fact that $V\geq \phi$. 
\end{proof}
Now that we have one inequality of the weak dynamic programming principle, it is only natural to characterize the value function with an HJB equation. For this, let us first introduce the \textit{Dynkin} operator
\begin{align*}
    \Lc^a (x,p,X) := b(x,a)^{\top} p + \frac{1}{2}Tr(\sigma \sigma^{\top} (x,a)X), ~~~ (a,x,p,X) \in U\times\R^d\times\R^d\times S^d.
\end{align*}
As is common in stochastic control, we do not expect the HJB equation to have a classical solution. For this reason, we ask for a weaker notion. For a definition and further discussion of viscosity solutions, we refer to the user's guide \cite{crandall1992user}. In the next results, for a given function $\varphi : [0,T]\times C\to \R$, we denote by $\varphi_\star$ and $\varphi^\star$ its lower-semicontinuous and upper-semicontinuous envelopes respectively.
\begin{proposition}\label{prop:SupersolutionProperty}
    Assume that $V$ is locally bounded on $(0,T) \times C$. Then the function $V_\star$ is a viscosity supersolution on $[0,T) \times C$ of
    \begin{align*}
        -\partial_t \varphi + \sup_{u \in U} \{ - \Lc^u (x,D\varphi, D^2\varphi) - f(x,u) \} =0.
    \end{align*}
\end{proposition}
\begin{proof}
The proof is analogous to the one of \cite[Theorem 4.2 (i)]{bouchard2012weak}, adapting it into a Lagrange formulation.
\end{proof}

\subsection{Subsolution property}
To prove the subsolution property one needs a feasible control to exist for the problem with open and closed set constraints, 
\emph{i.e.}, $\Uc^\circ(t,x) \not = \emptyset$ for each $(t,x) \in [0,T] \times C^\circ$, and $\Uc (t,x) \not = \emptyset$ for each $(t,x) \in [0,T] \times C$. Moreover, one needs such control to allow to switch to another admissible control in a measurable way. One condition that guarantees this property is the following.

\begin{assumption}\label{assump:lipschitz-control} 
    There exist two Lipschitz continuous mappings $\tilde{u} : C^\circ \to U$ and $\hat{u} : C \to U$ such that for all $(t,x) \in [0,T] \times C^\circ$, the solution $\tilde{X}^{t,x}$ of
    \begin{align*}
        \tilde{X}_s  = x + \int_t^s b(\tilde{X}_r,\tilde{u}(\tilde{X}_r)) \textnormal{d}r  + \int_t^s \sigma(\tilde{X}_r,\tilde{u}(\tilde{X}_r)) \cdot \textnormal{d}W_r,
    \end{align*}
    satisfies $\tilde{X}^{t,x}_s \in C^\circ$ for all $s\in [0,T], \P\text{-a.s.}$, and, for all $(s,y) \in [0,T]\times C$, the solution $\hat{X}^{t,x}$ of 
    \begin{align*}
        \hat{X}_s  = x + \int_t^s b(\hat{X}_r,\hat{u}(\hat{X}_r)) \textnormal{d}r  + \int_t^s \sigma(\hat{X}_r,\hat{u}(\hat{X}_r)) \cdot \textnormal{d}W_r
        ,
    \end{align*}
    satisfies $\hat{X}^{t,x}_s \in C$ for all $s\in [0,T], \P\text{-a.s.}$
\end{assumption}

We only use the first part of the Assumption to prove the subsolution property. The second part will be useful for the estimates needed for the value function.

\begin{proposition}\label{prop:PPD-inf}
    Let Assumption \ref{assump:lipschitz-control} hold true. For any $(t,x) \in [0,T] \times C^{\circ}$, for any $B \subseteq [0,T] \times C^\circ$ open neighborhood of $(t,x)$, for any $\nu \in \Uc_t$ and for any continuous function $\phi: \overline{B} \to \R$ satisfying $V^\circ \leq \phi$ on $\overline{B}$, we have 
    \begin{align*}
        V^\circ(t,x) \leq \E \left[ \int_t^{\tau^\nu} f(X^{\nu,t,x}_r , \nu_r) \mathrm{d}r + \phi(\tau , X^{\nu,t,x}_{\tau^\nu}) \right],
    \end{align*}
    where $\tau^\nu$ is the first exit time of $(s,X^{\nu,t,x}_s)_{s\geq t}$ from $B$.
\end{proposition}
\begin{proof} Since Assumption \ref{assump:lipschitz-control} holds true, we use a Mayer type reformulation of the problem and apply \cite[Lemma 4.9 (ii)]{bouchard2012weak} in order to get the desired result.
\end{proof}

\begin{proposition}\label{prop:SubsolutionProperty}
    Let Assumption \ref{assump:lipschitz-control} hold true and assume that $V^\circ$ is locally bounded on $[0,T) \times C^\circ$. Then $(V^\circ)^{\star}$ is a subsolution on $[0,T) \times C^\circ$ of 
    \begin{align*}
        - \partial_t \varphi + \sup_{u \in U} \left\{ - \Lc^u(x,D\varphi, D^2\varphi) - f(x,u) \right\} = 0.
    \end{align*}
\end{proposition}

\begin{proof}
Again, the proof is analogous to the one of \cite[Theorem 4.2 (ii)]{bouchard2012weak}, adapting it into a Lagrange formulation.
\end{proof}

\subsection{Link between both value functions}

Note that from the definition of both value functions, one has that $V(t,x) \leq V^\circ(t,x)$. Knowing that both value functions are supersolution and subsolution of the same constrained HJB equation, we can obtain the remaining inequality (and thus the equality of the value functions) from a comparison result. Nonetheless, the comparison result is not a direct result. In the literature, it is usually assumed that the functions $K,\sigma,b$ and $f$ are bounded \cite{ishii2002class} or that they have some regularity \cite{katsoulakis1994viscosity} that does not hold true in our electric market model. We will follow the ideas of \cite{bouchard2012weak}, where a regularity condition over the open value function is presented, that will allow us to establish the comparison result.


\begin{definition}\label{def:RO}
    Consider an open set $O\subseteq \R^d$ and a function $w : [0,T]\times \bar{O} \to \R$. We say that $w$ is of class $R(O)$ if for any $(t,x) \in [0,T) \times \partial O$
    \begin{itemize}
        \item[(i)] There exists $r>0$, an open neighborhood $B$ of $x$ in $\R^d$, and a function $l: \R_+ \to \R^d$ such that
        \begin{align}
           \label{ineq:defR(C)-1} &\liminf_{\varepsilon\to 0} \varepsilon^{-1} |l(\varepsilon)| < \infty,\\
           \label{in:defR(C)-2} &y+l(\varepsilon) + o(\varepsilon) \in O, ~\forall y \in O\cap B, \varepsilon \in (0,r).
        \end{align}
        \item[(ii)] There exists a function $\lambda:\R_+ \to \R_+$ such that
        \begin{align}
           \label{eq:defR(C)-3} & \lim_{\varepsilon \to 0} \lambda(\varepsilon) = 0, \\
            & \lim_{\varepsilon \to 0} w(t+\lambda(\varepsilon), x + l(\varepsilon)) = w(t,x). \textcolor{white}{-------} \label{eq:RO-continuous}
        \end{align}
    \end{itemize}
\end{definition}
The intuition behind this class of functions is that for each point at the boundary, there exists a trajectory along the interior such that the function is continuous following this trajectory. 

\begin{remark}
    This class of functions is highly dependent on the set $O$. For example, if $O$ is convex, then all continuous functions are of class $R(O)$.\footnote{This is direct by considering two points $x^1,x^2 \in O$ and taking $l:\R_+\to \R^d$ as $l(\varepsilon) = 2^{-1 }\varepsilon (x^1 + x^2) - x$ and $\lambda:\R_+ \to \R_+$ as $\lambda(\varepsilon) = \varepsilon$.}
\end{remark}

If one has a $\Cc^1$ characterization of the state constrained set, it is known that the conditions in Definition \ref{def:RO} hold true for the upper semicontinuous envelope of the open value function $V^\circ$, see \cite[Proposition 4.12]{bouchard2012weak}. However, such a smooth characterization does not hold in our setting, so we will introduce a better suited condition, which generalizes the classic \textit{inward-pointing} condition. In Remark \ref{remark:InwardPointing} we provide an economic interpretation of this assumption in the context of the electric market.

\begin{assumption}\label{assump:InwardPointing}
    For each $x \in \partial C$, there exist $u_x \in U$, $r>0$ and $\eta >0$ such that
\[
\sigma (z,u_x) = 0, \forall z \in B(x,\eta) \cap C \quad\text{and}\quad B(z+tb(z,u_x), rt) \subseteq C^\circ, \forall z \in B(x,\eta) \cap C, \forall t \in (0,r).
\]
\end{assumption}


In the next proposition, we adapt the ideas from \cite[Proposition 4.12]{bouchard2012weak} to a setting in which there is no smooth characterization of the set $C$.

\begin{proposition}\label{prop:VcircR(C°)}
    Let Assumption \ref{assump:InwardPointing} hold true. Then $(V^\circ)^\star$ is of class $R(C^\circ)$.
\end{proposition}

\begin{proof}
    Fix $(t,x) \in [0,T) \times \partial C$ and let $u_x\in U$, $r>0$ and $\eta>0$ be as in Assumption \ref{assump:InwardPointing}. Denote by $\hat{x}_z(\cdot)$ the unique solution\footnote{Existence and uniqueness follow since $b(\cdot,u_x)$ is globally Lipschitz.} of the Ordinary Differential Equation 
    \begin{align*}
        x'(s) =  b(x(s),u_x), ~x(0) = z.
    \end{align*}
    Due to the continuity of $b$, we have that, as $s\to 0$
    \begin{align}\label{eq:Proof-inward-1}
        s^{-1} \bigg| \int_0^s b(\hat{x}_x (r),u_x) \text{d}r\bigg| \to |b(x,u_x)|  \quad \text{and} \quad s^{-1} \bigg| \int_0^s \big( b(\hat{x}_x(r), u_x) - b(x, u_x) \big) \text{d}r \bigg| \to 0.
    \end{align}
    Now set $l(\varepsilon) := \hat{x}_x(\varepsilon) - x, ~~\lambda(\varepsilon) = \varepsilon$. We have $\lambda(\varepsilon) \to 0$ and, from the limit in the left of \eqref{eq:Proof-inward-1}, that $\liminf_{\varepsilon \to 0}\varepsilon^{-1}|l(\varepsilon)|< \infty$. Let us take $y \in C^\circ \cap B(x,\eta)$ and $\varepsilon>0$ small enough, then we have that
    \begin{align*}
        y + l(\varepsilon) + o(\varepsilon) &= y + \int_0^\varepsilon b(\hat{x}_x(r),u_x)\text{d}r + o(\varepsilon) \\
        &= y + \int_0^\varepsilon \big( b(\hat{x}_x(r), u_x) - b(x,u_x) \big) \text{d}r + \varepsilon b(x,u_x) + o(\varepsilon) \\
        &= y + \varepsilon b(x,u_x) + \hat{o}(\varepsilon),
    \end{align*}
    the last step by using the limit in the right of \eqref{eq:Proof-inward-1}. By taking $\hat{\varepsilon} \in (0,r)$ small enough, such that $\lVert \hat{o}(\varepsilon) \rVert \leq 2^{-1}r \varepsilon$ for each $\varepsilon \in (0,\hat{\varepsilon})$, by Assumption \ref{assump:InwardPointing}, we have that $y+ l(\varepsilon)+ o(\varepsilon) \in C^\circ$ for each $y \in C^\circ \cap B(x,\eta), \varepsilon \in (0,\hat{\varepsilon})$.
    
    Consider $(s,y) \in [0,T]\times C^\circ$ close to $(t,x)$. For $\varepsilon>0$ small enough, such that $\hat{x}_y(\xi) \in B(x,\eta)$ for each $\xi \in [0, \lambda(\varepsilon)]$, we take a control $\nu^\varepsilon$ such that
    \begin{align*}
        J(\nu^\varepsilon ; s+\lambda(\varepsilon), \hat{x}_y(\varepsilon)) \leq V^{\circ}(s+\lambda(\varepsilon), \hat{x}_y (\varepsilon)) + \varepsilon.
    \end{align*}
    Now, by setting the control
        $\tilde{\nu}^{\varepsilon} := 1_{[s,s+\lambda(\varepsilon)]} u_x + 1_{(s+ \lambda(\varepsilon), T]} \nu^\varepsilon$,
    since $\sigma(\cdot,u_x)=0$ on $B(x,\eta)$, we obtain that 
    \begin{align*}
        V^\circ(s,y) &\leq \E \left[ \int_s^{T} f(X^{\tilde{\nu}^\varepsilon,s,y}_r , \tilde{\nu}^\varepsilon_r) \text{d}r \right] \\
        & =  \int_s^{s+\lambda(\varepsilon)} f(\hat{x}_y(r), u_x)\text{d}r + \E \left[ \int_{s+ \lambda(\varepsilon)}^T f(X^{\nu^\varepsilon, s+ \lambda(\varepsilon), \hat{x}_y(\varepsilon)}) \right] \\
        & \leq o(1) + V^{\circ}(s+\lambda(\varepsilon), \hat{x}_y(\varepsilon)) + \varepsilon.
    \end{align*}
    Let us consider $\tau = \inf\{s \geq 0, \hat{x}_x(s) \in B(x,\eta)^c\}$ and let us take $\varepsilon \in (0,\tau)$. Due to the continuity of ODEs with respect to the initial data, see for instance \cite[Theorem 7.4]{avez2020differential}, we have that $\hat{x}_y(\xi) \in B(x,\eta)$ for $\eta \in [0,\varepsilon]$ for $y$ close enough to $x$.  By taking $y \to x$, we have that $\hat{x}_y(\varepsilon) \to \hat{x}_x(\varepsilon)$. Then taking $\limsup$ of $(s,y) \to (t,x)$, we have
    \begin{align*}
            (V^\circ)^\star (t,x) \leq (V^\circ)^\star(t+\lambda(\varepsilon), \hat{x}_x(\varepsilon)) + o'(1) = (V^\circ)^{\star}(t+ \lambda(\varepsilon), x + l(\varepsilon)) + o'(1).
    \end{align*}
    For $\varepsilon$ small, finally this implies that
    \begin{align*}
        (V^\circ)^\star(t,x) & \leq \liminf_{\varepsilon \to 0} (V^\circ)^{\star} (t+\lambda(\varepsilon), x + l(\varepsilon)) \leq\limsup_{\varepsilon \to 0} (V^\circ)^{\star} (t+\lambda(\varepsilon), x + l(\varepsilon)) 
         \leq (V^\circ)^\star(t,x).
    \end{align*}
    Hence we obtain the continuity condition \eqref{eq:RO-continuous} for the function $V^\circ$.
\end{proof}


\begin{proposition}\label{prop:PolynomialGrowth}
    Let Assumption \ref{assump:lipschitz-control} hold true, then both $V$ and $V^\circ$ have quadratic growth.
\end{proposition}
\begin{proof}
  We know from the standard estimates for the strong solutions of SDEs, see for instance \cite[Section 3.2.]{zhang2017backward}, the existence of constants $C_1,C_2>0$ such that
\[
|V(t,x)| \leq C_1 (1+|x|^2), ~|V^\circ(t,x)| \leq C_2 (1+|x|^2).
\]  
\end{proof}
Now that we have the comparison principle, we can state the complete result.

\begin{theorem}\label{th:HJB-constrained}
    Let Assumptions \ref{assump:lipschitz-control} and \ref{assump:InwardPointing} hold true. One has that $V = V^\circ$ on $[0,T] \times C^\circ$, $V$ is continuous and is the unique constrained viscosity solution on $[0,T) \times C$ of
    \begin{align}\label{eq:ConstrainedHJB}
        - \partial_t \varphi + \sup_{u\in U} \left\{ - \Lc^{u}(\cdot, D\varphi,D^2\varphi) - f(\cdot,u) \right\} = 0, \quad \varphi(T,\cdot) = 0,
    \end{align}
    in the class of the functions with polynomial growth and such that its upper semicontinuous envelope is of class $R(C^\circ)$.
\end{theorem}
\begin{proof}
    Recalling that $V \leq V^\circ$ on $[0,T] \times C^\circ$, one has that $V^\star \leq (V^\circ)^\star$ on $[0,T] \times \overline{(C^\circ)} = [0,T] \times C$. Now by Propositions \ref{prop:SupersolutionProperty}, \ref{prop:SubsolutionProperty}, \ref{prop:PolynomialGrowth} we have that $V$ and $V^\circ$ are supersolution and subsolution of \eqref{eq:ConstrainedHJB} respectively. Applying Theorem \ref{Teo:Comparison}, we get that $(V^\circ)^\star \leq V_\star$ on $[0,T] \times C$. Therefore, we have
    \[
    V^\circ \leq (V^\circ)^\star \leq V_\star \leq V,
    \]
    hence obtaining the equality $V=V^\circ$ and the continuity of the functions.
\end{proof}

\section{Solving the planning problem in the case of Itô diffusions}\label{sec:ito-process}


In this section we apply the results of Section \ref{sec:State-Constr} to solve Problem \eqref{prob:equivalent-problem}. The electric market has several sources of uncertainty that include, among others, the randomness that comes from the volatility in the production and the different amount of production capacity depending on the time and season in which one is producing. The model is able to tackle in some way both types of uncertainties, but we will focus on the first one.
    
From now on, let us separate the control into two, the investment plan which is restricted to a maximum budget $\bar{\mu}>0$ and a stabilization control which is an indicator at each node of what proportion of the total volatility is being stabilized. We consider thus
\begin{align*}
 U = U_{\bar \mu} \times [0,1]^N, \quad \text{ with } \quad   U_{\bar \mu} := \bigg\{ \mu \in \R^N_+ :  \sum_{i=1}^N h^i_\mu\mu^i \leq \bar{\mu} \bigg\}, 
\end{align*}
    where $h^i_\mu \in (0,\infty)$ represents the unitary cost of the technology used to increase the capacity at node $i \in \{1,\dots,N\}$. Additionally, we consider $h^i_\alpha \in (0,\infty)$ the stabilization cost at the node $i\in \{1,\dots,N \}$. We do not limit the total spend in stabilization, given that it is an emergency purposed control. Nonetheless, we limit the investment by a fixed budget $\bar{\mu}$ at each instance $t \in [0,T]$, and since we consider a continuation of investment, taking a year as the unit of time, we can interpret $\bar{\mu}$ as the annual budget for investment in renewable energies.

For $\nu = (\mu,\alpha) \in \Uc$, let $(Q^{\nu,t,x}, D^{t,x})$ be the unique strong solution of the following SDE
\begin{align*}
    Q^i_s = x_Q^i + \int_t^s \mu_\tau^i \text{d}\tau + \int_t^s \sigma^i (Q^i_\tau - \hat{Q}^i) (1-\alpha^i_\tau)\text{d}W_\tau, \quad
    D^i_s = x_D^i + \int_t^s p^i_D(\tau)\text{d}\tau,\quad i=1,\dots,N,
\end{align*}
With \green{$x= (x_Q,x_D) \in \R^N\times \R^N$ the initial condition,} $\sigma^i\geq0$ for $i=1,\dots,N$, the volatility rate at each node,  $\hat{Q}\in \R^N_+$ a fixed amount of nonstochastic capacity,
and $p_D : \R_+ \to \R^N$ a bounded Lipschitz continuous function. We then consider the (demand-independent) cost of investment functions $h:U \times \R^N \to \R^N$ as
\begin{align*}
    h^i((\mu,\alpha),Q) = h^i_\mu \mu^i + h^i_\alpha \sigma^i(Q^i-\hat{Q}^i) \alpha^i, \quad i=1,\dots,N,
\end{align*}
which gives us all the elements for \eqref{prob:equivalent-problem} to be well defined. The presence of $\hat{Q}$ represents a previously available non-renewable source, whereas all the additional capacity is renewable energy, therefore the volatility depends only on the extra amount $(Q-\hat{Q})$.
    
Let us take a look at the dynamics and discuss how to use the tools presented in Section \ref{sec:State-Constr}. The drift term in the dynamics of $Q$ represents the acquisition of technology, which is linked to the investment budget, and we assume that it is normalized. The volatility term represents a fixed unitary volatility, considering only the renewable part of the capacity, and it is stabilized through the control $\alpha$. While $Q$ is a controlled capacity, $D$ is a deterministic prediction
of the demand for a period of time. We need a model such that in critical scenarios it is possible to return to the interior of the constrained set. 
This is a condition that links the planification $p_D$ and the overall investment budget $\bar{\mu}$. 

\begin{remark}\label{remark:InwardPointing}
    In the electricity market context, Assumption \ref{assump:InwardPointing} has an economic interpretation. What does it mean to be in the boundary of the set $K$? At these points the capacity is in a critical situation in which one has to produce at maximum capacity just to fulfill the demand. Moving just a little in the wrong direction puts us outside of the set $K$, which would mean not fulfilling the demand, which is catastrophic for the market (and potentially illegal). Therefore, whatever the cost, the network needs to go back to a comfortable state. For this reason, payments are made to turn off the volatility, for instance by buying hydraulic energy, and exit the critical state. Mathematically, in our model, this translates into turning off the volatility by taking $\alpha\equiv1$, and investing 
    in an emergency plan. For instance, forgetting about the transmission lines and producing the extra energy locally at each node.
\end{remark}

\begin{assumption}\label{assump:IP-constant-A}
     The budget $\overline{\mu}$ satisfies
     \begin{align*}
         \overline{\mu} > \sup_{s \in [0,T]} \left( \sum_{i=1}^N h^i_\mu (p_D^{i} (s))^+ \right).
     \end{align*}
\end{assumption}


\begin{proposition}\label{prop:5 implies 3,4}
    Let Assumption \ref{assump:IP-constant-A} hold true. Then Assumptions \ref{assump:lipschitz-control} and \ref{assump:InwardPointing} hold true.
\end{proposition}
\begin{proof}
    We start by considering Assumption \ref{assump:InwardPointing}. Let us consider $(s,Q,D) \in \R_+ \times  K$. Due to the strict inequality in Assumption \ref{assump:IP-constant-A}, there exist $\varepsilon>0$ such that
    \begin{align*}
        \overline{\mu} > \sup_{s \in [0,T]} \left( \sum_{i=1}^N h^i_\mu [(p_D^{i} (s))^+ + 2\varepsilon] \right).
    \end{align*}
    Then, by taking $\mu^i = (p^i_D(s))^+ + 2\varepsilon$ for each $i=1,\dots,N$, we have that $\mu \in U_{\mu}$ and $p_D^i(s) + \varepsilon < \mu^i,  \forall i = 1,\dots, N. $
    By taking $u = (\mu,\textbf{1})\in U$ the volatility turns off, which gives us the first part of Assumption \ref{assump:InwardPointing}. What is left to check is the interior pointing condition. For $t>0$, note that we have  
    \begin{align*}
        G(Q+t\mu^i, D+tp_D(s))  &= \max_{\phi \in [\underline{\phi},\overline{\phi}]} \min_i \left( (Q^i + t\mu^i) - (D^i + tp^i_D(s)) + \sum_{e \in K_i} \left( sgn(i,e)\phi^e - \frac{r^e}{2}(\phi^e)^2 \right) \right) \\
        & \geq  G(Q,D)+t\varepsilon
         \geq t\varepsilon 
    \end{align*}
    By taking $r=\varepsilon/4$, by an analogous calculation for the function $G$, we have that the second condition for Assumption \ref{assump:InwardPointing} holds true. 
    
    To see that Assumption \ref{assump:lipschitz-control} holds, we need to prove the existence of a predictable control $\nu$ such that the process $(Q^\nu,D)$ stays in  $K^\circ$. For $\eta>0$ small enough, consider the control $\nu(s,Q,D) := (p_D(s)+ \eta)$, and any $(t,Q_t,D_t) \in [0,T]\times K^\circ$. Then, by direct calculation, we get
$
        G(Q^\nu_s, D_s) \geq (s-t)\eta + G(Q_t,D_t)> (s-t)\eta \geq 0,
$
    obtaining that the trajectory stays the whole period in the interior of $K$. Since $\nu$ is a Lipschitz control with respect to the triplet $(s,Q,D)$, Assumption \ref{assump:lipschitz-control} holds true. 
\end{proof}

Now that the chosen dynamics fits the framework of Section \ref{sec:State-Constr}, what is left to check is that the objective function is also well suited. With this in mind, we state the following assumption.

\begin{assumption}\label{assump:C(Q,D) Lipschitz}
    (i) Assumption \ref{assump:1} holds true and the function
        $C(Q,D) = \sum_{i=1}^N c^i(q^{\star,i}(Q,D),Q^i)$,
    is  a globally Lipschitz continuous function on $K$. 
    
    (ii) $h$ is globally Lipschitz continuous with respect to $(Q,D)$ in $K$, uniformly in $(\mu,\alpha)$.
\end{assumption}

\begin{remark}
    The Lipschitz continuity assumption for the problem \eqref{prob:DAPQD} is intuitive, as a change in capacity or demand only requires a relocation of the energy plus a little change in the amounts of production. The graph structure of the problem makes it complicated to prove in the general case, but in specific cases it is not hard to check. For example, in Section \ref{sec:Num-Results} the result comes directly from the Lipschitz continuity of \eqref{func:PWlinear} and the smoothness of the flow constraint.
\end{remark}


Let $H_{CE}  : \R_+ \times  \R^N\times  \R^N \times  \R\times \R^N\times  \R^N\times S^{N} \to\R$ be the Hamiltonian for the problem of CE
\begin{align*}
    H_{CE}(t,x,y,r,w,v,X)  := \sup_{u \in U} \left\{ - r - \sum_{i=1}^N \left [u_\mu^i w^i + p_D^i(s)v^i  + \frac{1}{2} (\sigma^i (1-u_{\alpha}^i)(x^i-\hat{Q}^i))^2 X^{i,i} + h^i(u,x,y) \right] \right\}.
\end{align*}
Following the results of the previous section, we have the following verification result.
\begin{theorem} \label{Teo:HJB_CE}
    Let Assumptions \ref{assump:IP-constant-A} and \ref{assump:C(Q,D) Lipschitz} hold true. One has that $V = V^\circ$ on $[0,T] \times K^\circ$, $V$ is continuous and is the unique constrained viscosity solution on $[0,T)\times K $ of \begin{align*}
        -\partial_t \varphi + H_{CE}(s,x,y,\partial_s\varphi, D_x \varphi, D_y \varphi, D^2_x\varphi) - \sum_{i=1}^N c^i(q^\star (x,y),x) = 0,& \quad (t,s,x,y) \in [0,T) \times \R \times K\\
        \varphi(T,s,x,y) = 0, & \quad (s,x,y)\in \R \times K
    \end{align*}
    in the class of functions with polynomial growth and having an upper semicontinuous envelope of class $R(K^\circ)$.
\end{theorem}
\begin{proof}
    Due to Proposition \ref{prop:5 implies 3,4}, we can apply Theorem \ref{th:HJB-constrained} to obtain the desired result.
\end{proof}

\section{Numerical results}\label{sec:Num-Results}

In this section, we numerically solve the HJB equation that arises from Theorem \ref{Teo:HJB_CE} in a simplified market model. To this end, we consider \nico{three} settings using the three-node representation of the Chilean electric market shown in Figure \ref{graph:Chile}. The first setting assumes constant demand, that is, $p_D \equiv 0$. The second \nico{and third settings} assume a linear increase in demand, that is, $p_D \equiv \iota > 0$. The settings reflect short-term, mid-term and long-term planning, respectively. The first corresponds to a period of two years, where demand variations are negligible and a constant demand is appropriate. The second one covers a period of one decade \nico{and the third one of twenty years}, so the increase in demand must be considered. \nico{The main difference between the last two settings is that, given the initial capacities in the network, the state constraint activates only in the long-term.} We calibrate the simulations to mimic the Chilean data presented in \cite{hernandez2023pollution}. At each node, we consider two non-renewable technologies, coal and gas, and one renewable technology, so that each node has a piecewise linear cost function of the form
\begin{align}\label{func:PWlinear}
    c^i(q,Q) = \int_0^{q^i} \big( c^i_R \cdot 1_{[0,Q^i-\hat{Q}^i] } (s) +  c^i_{\text{coal}} \cdot 1_{(Q^i - \hat{Q}^i,Q^i - Q^i_{\text{gas}})}(s) + c^i_{\text{gas}} \cdot 1_{[Q^i_{\text{gas}}, \infty)}(s) \big)  \mathrm{d}s,
\end{align}
for each $i =1,\dots, N$. The marginal costs satisfy $0<c_R^i< c_{\text{coal}}^i<c^i_{\text{gas}}$, the existing capacities $Q^i_\text{coal},Q^i_\text{gas}  \geq 0$, and $  \hat{Q}^i = Q^i_\text{coal}+ Q^i_\text{gas} $. This cost function represents three distinct unit costs corresponding to renewable energy, coal, and gas. It is important to note that $Q_0 \geq \hat{Q}$, which means that the initial capacity includes an existing renewable capacity given by $Q_{\text{renewable}} := Q_0 - \hat{Q}$.

For easy reference, we present in Table \ref{table:units} the units used for the parameters and functions in the model. The time unit is one hour, which explains the units of the instantaneous quantities.

\begin{table}[h]
    \centering
    \small
    \begin{tabular}{|c|c||c|c|} \hline
            \bf Quantity  & \bf Units  & \bf Quantity & \bf Units      \\ \hline
            $Q^i_t$  &  MW     &  $\mu^i_t$     &  $\text{MW} \times  (\text{hours})^{-1}$  \\ \hline
            $D^i_t$  &  MW     &  $h^i_\mu$    &    dollars $\times$ $(\text{MW})^{-1}$   \\ \hline
            $q^i_t$  &  MW     &   $\bar{\mu}$   & dollars $\times $  $(\text{hours})^{-1}$    \\ \hline
            $\phi^e_t$  & MW     &   $\alpha^i_t$    &  scalar    \\ \hline
            $r^e$ & $(\text{MW})^{-1}$    &   $h^i_\alpha$   &   $\text{dollars} \times (\text{MW} \times (\text{hours})^{\frac{1}{2}})^{-1}$  \\ \hline
            $c^i_R, c^i_{\text{coal}},c^i_{\text{gas}} $ & $ \text{dollars} \times (\text{MW} \times 
 \text{hours})^{-1}$     &  $\sigma^i$   &  $(\text{hours})^{\frac{1}{2}}$     \\ \hline
    \end{tabular}
    \caption{Units in the model}
    \label{table:units}
\end{table}

In Subsections \ref{subsec:st-planning}, \ref{subsec:lt-planning} and \ref{subsec:slt-planning}, we plot 50 trajectories of the main processes. We prefer to display, instead of the behavior of individual trajectories, their distribution over the entire periods.


\subsection{Short-term planning} \label{subsec:st-planning}

In this subsection, we present simulations for the problem with a horizon of $T = 2$ years, assuming a constant demand process at each node, as the changes are negligible over this period. The initial conditions used for this simulation are summarized in Table \ref{table:constant-demand-initial}. We also set the flow limit at each edge to $\phi^e \in [-6 \text{ GW}, 6 \text{ GW}]$. Regarding the dynamics, we fix the total maximum budget to $\bar{\mu} T = 9 \cdot 10^6$ dollars for the two-year period, and set the volatility at each node to $\sigma^i = 5\%$. The cost parameters chosen for this scenario are given in Table \ref{table:costs-constant-demand}.\footnote{The investment costs are based on the following IRENA report: \url{https://www.irena.org/-/media/Files/IRENA/Agency/Publication/2024/Sep/IRENA_Renewable_power_generation_costs_in_2023_executive_summary.pdf}} 

\begin{table}[h]
    \small 
    \centering
    \begin{tabular}{|c||c||c||c||c||c|}
    \hline
     \bf Location & $Q_0$ & $Q_{\text{coal}}$ & $Q_{\text{gas}}$ & $Q_{\text{renewable}}$ & $D_0$  \\
     \hline
     \bf North & $6000$ & $1800$& $2400$ & $1800$ & $3000$ \\
     \hline
     \bf South & $2000$ & $800$ & $1000$ & $200$ & $1000$  \\
     \hline
     \bf Center & $12000$ & $1200$ & $8400$ & $2400$ & $6000$  \\
     \hline 
    \end{tabular}
    \caption{Initial conditions in MW }
    \label{table:constant-demand-initial}

\end{table}


\begin{table}[h]
    \centering
    \begin{tabular}{|c||c||c||c||c||c|}
    \hline
     \bf Location & $c_R$ & $c_{\text{coal}}$ & $c_{\text{gas}}$ & $h_\mu$ & $h_\alpha$  \\
     \hline
     \bf North & $ 1    $ & $40 $ &  $80 $  &  $758 $   & $100 $  \\
     \hline
     \bf South & $1 $ & $40 $  &  $80 $  &   $2806 $  & $100 $  \\
     \hline
     \bf Center & $1 $ & $40 $ &  $80 $  &  $758 $   & $100 $ \\
     \hline 
    \end{tabular}
    \caption{Investment and Operational costs
    }
    \label{table:costs-constant-demand}
\end{table}

As shown in Figure \ref{fig:st-capacities}, the capacities in both the North and South remain largely unchanged throughout the entire period, with only minor fluctuations due to the inherent volatility of their stochastic dynamics. In contrast, the capacity in the Center increases over time. The controls $\nu = (\mu, \alpha)$, displayed in Figure \ref{fig:st-combined-nu}, show that investment is concentrated in the first half of the period, while no stabilization effort is applied. This pattern in the investment control is expected since the necessary condition for minimizing the Hamiltonian requires that the installation cost of renewable capacity, $h^i_\mu$, be lower than the marginal value of additional energy. Once this condition no longer holds, further investment is no longer optimal.

\begin{figure}[h]
    \centering
    \includegraphics[width=0.4\linewidth]{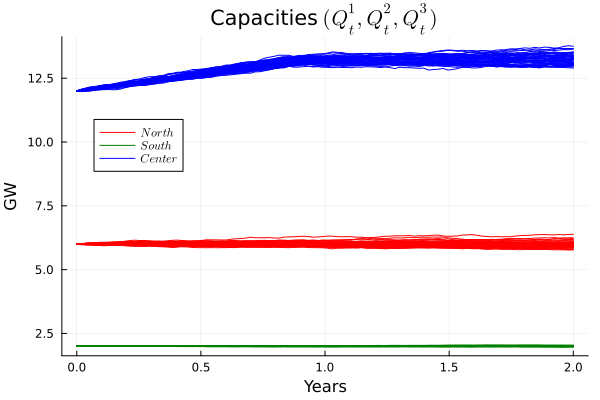}
    \caption{Short-term Capacities}
    \label{fig:st-capacities}
\end{figure}

\begin{figure}[h]
    \centering
    \begin{subfigure}[b]{0.33\linewidth}
        \centering
        \includegraphics[width=\linewidth]{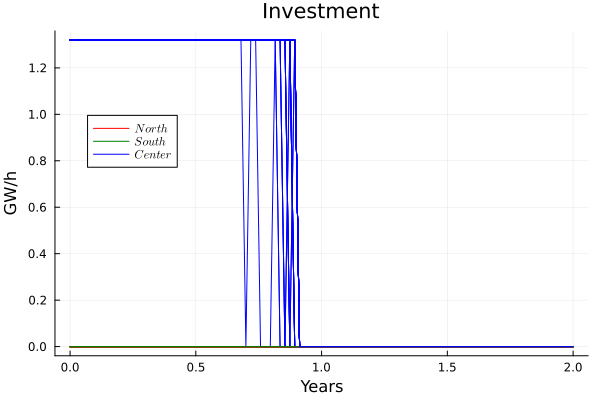}
        \caption{Short-term investment}
        \label{fig:st-investment}
    \end{subfigure}\quad 
    \begin{subfigure}[b]{0.33\linewidth}
        \centering
        \includegraphics[width=\linewidth]{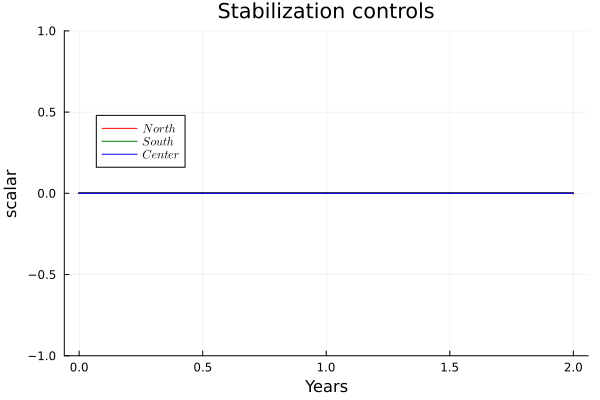}
        \caption{Short-term stabilization}
        \label{fig:st-stabilization}
    \end{subfigure}
    \caption{Short-term control $(\nu)$}
    \label{fig:st-combined-nu}
\end{figure}

The absence of stabilization spending, as shown in Figure \ref{fig:st-stabilization}, is unsurprising in this context. As discussed in Remark \ref{remark:InwardPointing}, stabilization becomes relevant only in critical scenarios. However, in this case, the capacity, especially the non-stochastic component, greatly exceeds the demand at each node and keeps the system far from any critical threshold. In fact, each node is initially equipped to cover twice its demand. Under such non-critical conditions, stabilization represents a significant cost without providing any tangible benefit because the value of $\alpha$ has no effect on the expected cumulative production cost.

\begin{figure}[h]
    \centering
    \begin{subfigure}[b]{0.33\linewidth}
        \centering
        \includegraphics[width=\linewidth]{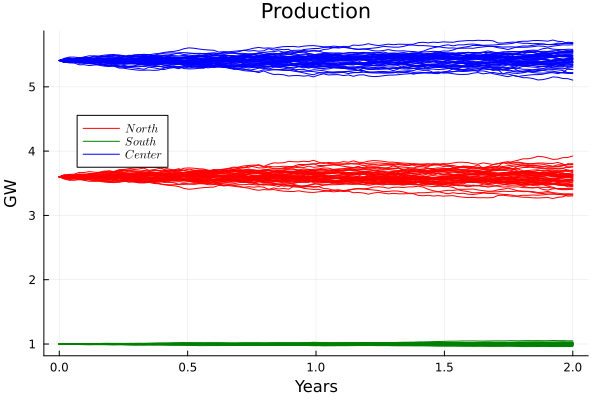}
        \caption{Short-term production}
        \label{fig:st-production}
    \end{subfigure}
    \quad
    \begin{subfigure}[b]{0.33\linewidth}
        \centering
        \includegraphics[width=\linewidth]{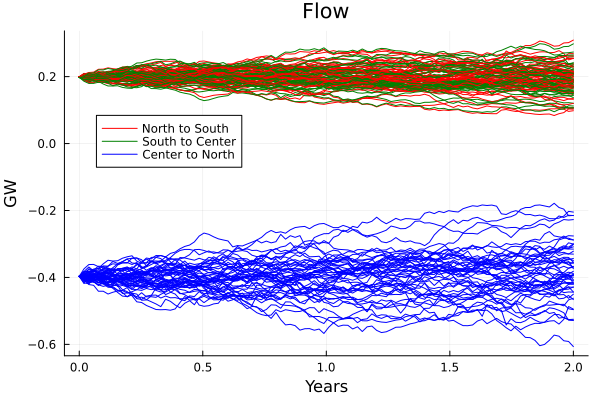}
        \caption{Short-term flow}
        \label{fig:st-flow}
    \end{subfigure}
    \caption{Short-term control $(q,\phi)$}
    \label{fig:st-combined-qphi}
\end{figure}

The solution of \eqref{prob:nu-subproblem} under the given control $\nu = (\mu, \alpha)$ is shown in Figure \ref{fig:st-combined-qphi}. Additionally, Figure \ref{fig:st-types-of-production} shows the level of production by energy source. The operational control $(q, \phi)$ appears steady, with fluctuations arising only from the volatility inherited from the stochastic dynamics of the capacity. This illustrates that the changes in the network are local, in the sense that there are no real changes in the flow, only in the type of production. This is even clearer in Figures \ref{fig:st-renewable} and \ref{fig:st-gas}, where we observe that the optimal strategy consists of a transition from gas to renewable energy as the source of production.

\begin{figure}[h]
    \centering
    \begin{subfigure}[b]{0.30\linewidth}
        \centering
        \includegraphics[width=\linewidth]{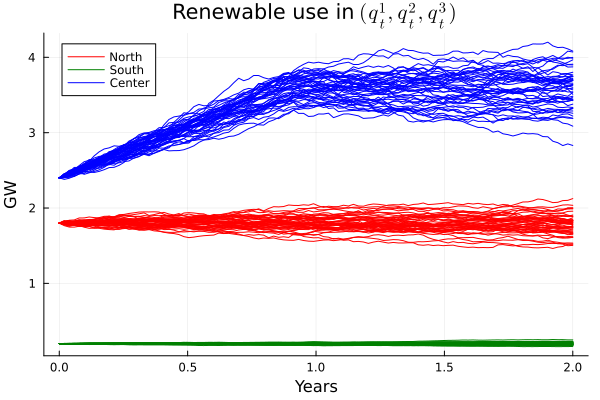}
        \caption{Renewable}
        \label{fig:st-renewable}
    \end{subfigure}
    \hfill 
    \begin{subfigure}[b]{0.30\linewidth}
        \centering
        \includegraphics[width=\linewidth]{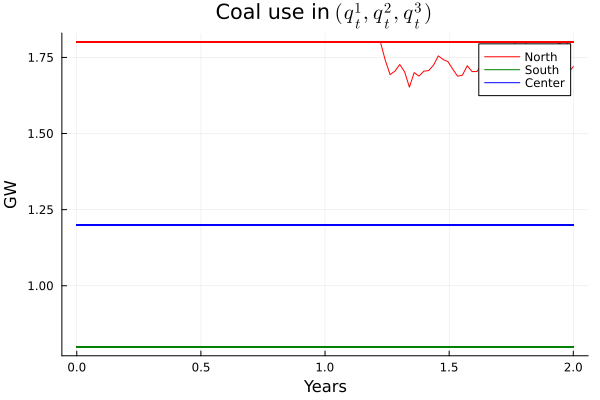}
        \caption{Coal}
        \label{fig:st-coal}
    \end{subfigure}
    \hfill 
    \begin{subfigure}[b]{0.30\linewidth}
        \centering
        \includegraphics[width=\linewidth]{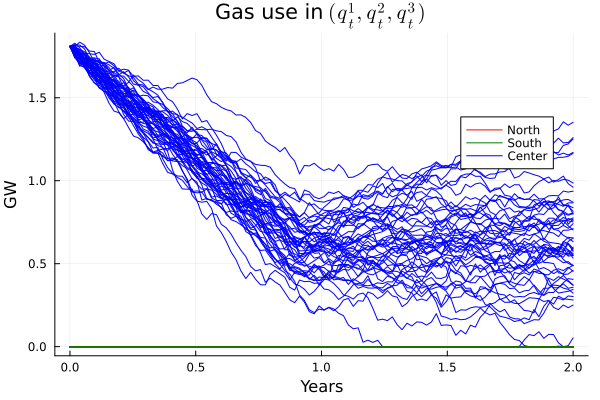}
        \caption{Gas}
        \label{fig:st-gas}
    \end{subfigure}
    \caption{Short-term types of production}
    \label{fig:st-types-of-production}
\end{figure}

The behavior of both investment and operational planning is a consequence of the short time period, as there is not enough time to achieve a structural change in the network’s overall energy transmission. 

\subsection{Mid-term planning} \label{subsec:lt-planning}

In this subsection we present simulations for the problem with $T = 10$ years and a linear increase in demand from $D_0$ to $D_{10} = \frac{53}{43} D_0$, which mimics an annual increase of $2\%$ per year. The initial conditions given in Table \ref{table:constant-demand-initial} are also used in this setting, along with the flow limit $\phi^e \in [-6 \text{ GW}, 6 \text{ GW}]$ at each edge. For the dynamics, we maintain the annual budget and adjust it to the 10-year period, resulting in a total budget of $\bar{\mu} T = 4.5 \cdot 10^7$ dollars. The volatility is kept at $\sigma^i = 5\%$ at each node. The cost parameters are the same as in the previous simulation, shown in Table \ref{table:costs-constant-demand}.

\begin{figure}[h]
    \centering
    \includegraphics[width=0.4\linewidth]{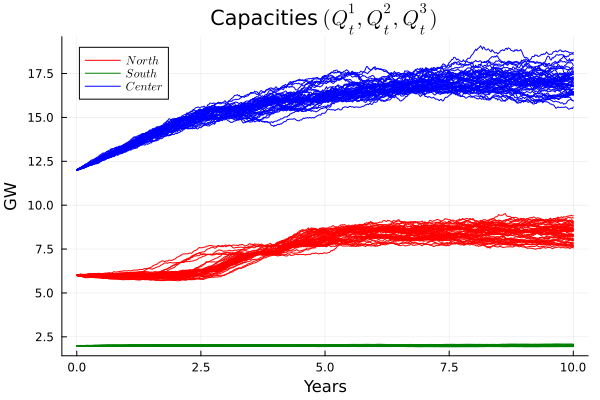}
    \caption{Mid-term Capacities}
    \label{fig:lt-capacities}
\end{figure}

Figure \ref{fig:lt-capacities} shows a sustained increase in capacity at both the North and Center nodes. In contrast, capacity in the South remains unchanged throughout the entire period. This is expected, as investing in new technology at that node is substantially more expensive. The controls $\nu = (\mu, \alpha)$ shown in Figure \ref{fig:lt-combined-nu} indicate sustained investment, shared between the North and Center until the eighth year, and \nico{as in the short-term case}, no stabilization effort is applied. Regarding the investment control, we observe an initial preference for investment in the Center. This is consistent with the previous case, where a local transition from gas to renewable sources was needed. Once this transition is complete, investment alternates between the North and Center. The absence of stabilization spending is expected for the same reasons as in the short-term simulation.

\begin{figure}[h]
    \centering
    \begin{subfigure}[b]{0.33\linewidth}
        \centering
        \includegraphics[width=\linewidth]{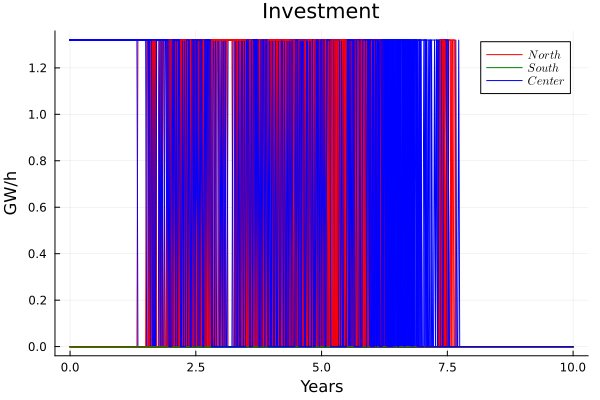}
        \caption{Mid-term investment}
        \label{fig:lt-investment}
    \end{subfigure}
    \quad
    \begin{subfigure}[b]{0.33\linewidth}
        \centering
        \includegraphics[width=\linewidth]{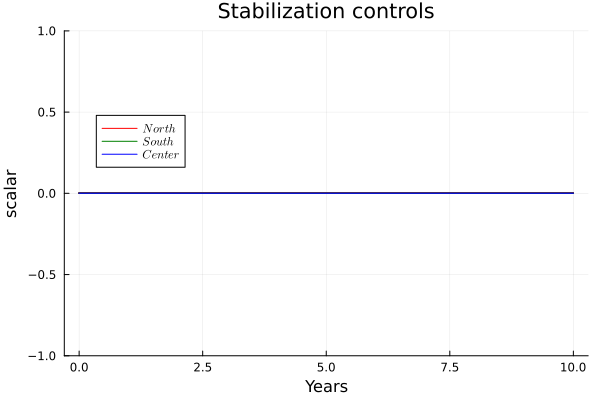}
        \caption{Mid-term stabilization}
        \label{fig:lt-stabilization}
    \end{subfigure}
    \caption{Mid-term control $(\mu,\alpha)$}
    \label{fig:lt-combined-nu}
\end{figure}

The continuous solution of \eqref{prob:nu-subproblem} under the given control $\nu = (\mu, \alpha)$ is shown in Figure \ref{fig:lt-combined-qphi}, and the breakdown of production by energy source is shown in Figure \ref{fig:lt-types-of-production}. We observe that the production control undergoes significant changes throughout the period, which can be divided into three main phases.

The first phase begins with an increase in production at the Center, accompanied by a brief decrease in production at the North. As shown in Figure \ref{fig:lt-flow}, this is due to a reduced need for flow in the Center. Once this transition is complete, the second phase starts, close to the beginning of the second year. During this phase, production at the North increases gradually, while internally there is a complete transition from coal to renewable sources. This can be seen in Figures \ref{fig:lt-renewable} and \ref{fig:lt-coal}. The transition occurs simultaneously at both the North and Center nodes, and the phase concludes with an almost complete shift to renewable energy at both locations.

The third phase begins around the fourth year and is characterized by an increase in flows toward the South. Midway through this phase, the investment control $\mu$ is turned off, so the capacity and production processes are influenced only by the volatility inherent in the dynamics.

In summary, the first phase corresponds to the transition from gas to renewable energy at the Center; the second phase to the transition from coal to renewable energy at both the North and Center; and the third phase to the transition from coal to external renewable sources at the South. As shown in Figure \ref{fig:lt-combined-qphi}, the uncertainty associated with renewable energy sources translates into uncertainty in the operational planning solution. This is expected since, in the hierarchy of controls, the operational component responds to the investment through the Day-ahead problem.

\begin{figure}[h]
    \centering
    \begin{subfigure}[b]{0.33\linewidth}
        \centering
        \includegraphics[width=\linewidth]{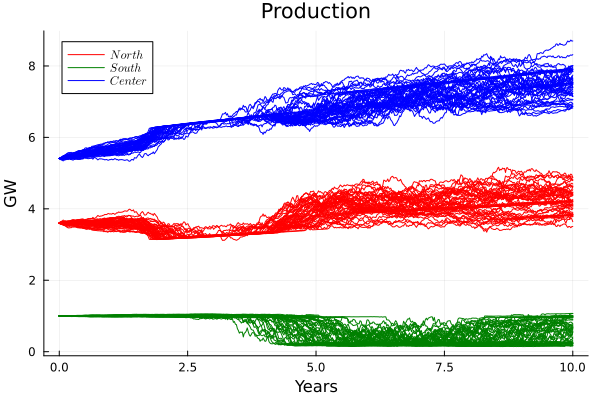}
        \caption{Mid-term production}
        \label{fig:lt-production}
    \end{subfigure}
    \quad
    \begin{subfigure}[b]{0.33\linewidth}
        \centering
        \includegraphics[width=\linewidth]{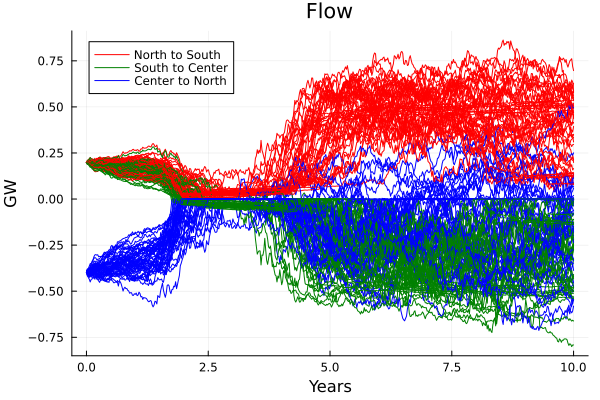}
        \caption{Mid-term flow}
        \label{fig:lt-flow}
    \end{subfigure}
    \caption{Mid-term control $(q,\phi)$}
    \label{fig:lt-combined-qphi}
\end{figure}
\begin{figure}[h!]
    \centering
    \begin{subfigure}[b]{0.30\linewidth}
        \centering
        \includegraphics[width=\linewidth]{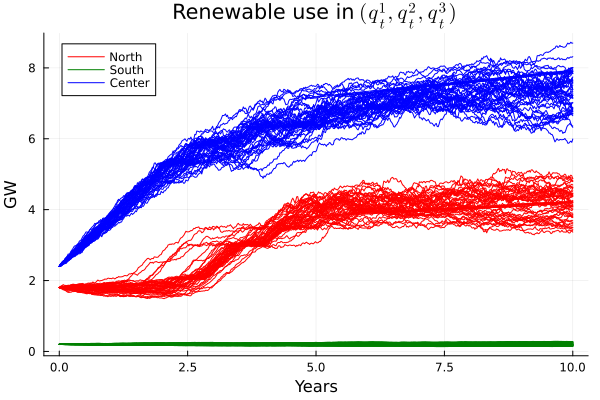}
        \caption{Renewable}
        \label{fig:lt-renewable}
    \end{subfigure}
    \hfill 
    \begin{subfigure}[b]{0.30\linewidth}
        \centering
        \includegraphics[width=\linewidth]{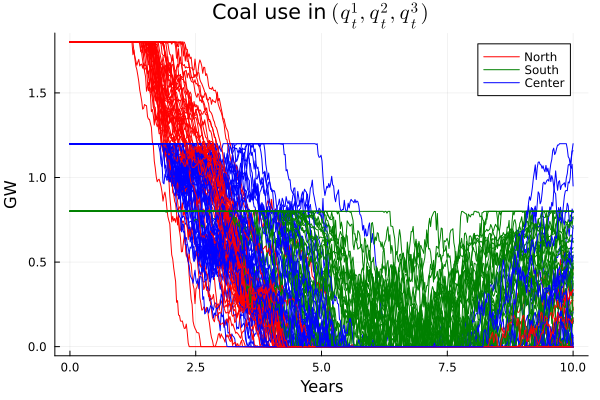}
        \caption{Coal}
        \label{fig:lt-coal}
    \end{subfigure}
    \hfill 
    \begin{subfigure}[b]{0.30\linewidth}
        \centering
        \includegraphics[width=\linewidth]{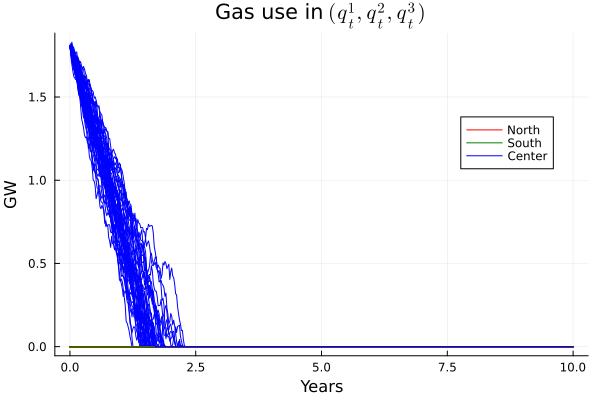}
        \caption{Gas}
        \label{fig:lt-gas}
    \end{subfigure}
    \caption{Mid-term types of production}
    \label{fig:lt-types-of-production}
\end{figure}

The overall behavior of the network in \nico{the mid-term} setting is consistent with what we expect from a massive entry of renewable energy. In this case, the network’s internal transmission changes, transforming the electric market into a system that depends primarily on renewable sources.

\subsection{Long-Term Planning} \label{subsec:slt-planning}

In this subsection, we present simulations for the problem with $T = 20$ and a linear increase in demand from $D_0$ to $D_{20} = \frac{37}{17} D_0$. This growth, higher than the annual increase considered in the mid-term planning, activates the state constraint of the problem, since the initial capacities reported in Table \ref{table:constant-demand-initial} are insufficient to meet the final demand. The remaining data are the same as in Subsection \ref{subsec:lt-planning}, with the cost parameters presented in Table \ref{table:costs-constant-demand}.
\begin{figure}[h]
    \centering
    \includegraphics[width=0.4\linewidth]{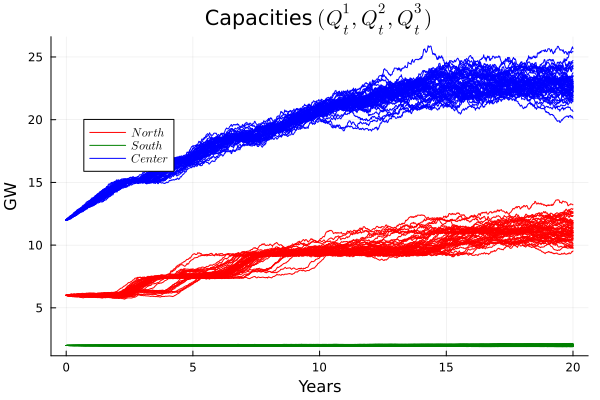}
    \caption{Long-term Capacities}
    \label{fig:slt-capacities}
\end{figure}

As shown in Figure \ref{fig:slt-capacities}, the evolution of capacity exhibits behavior similar to that in the previous setting, with a sustained increase at the North and Center nodes, while capacity in the South remains essentially constant, varying only due to stochastic fluctuations. This outcome is again expected, since investment costs in the North and Center are lower than in the South, see Table \ref{table:costs-constant-demand}.
\begin{figure}[h]
    \centering
    \begin{subfigure}[b]{0.33\linewidth}
        \centering
        \includegraphics[width=\linewidth]{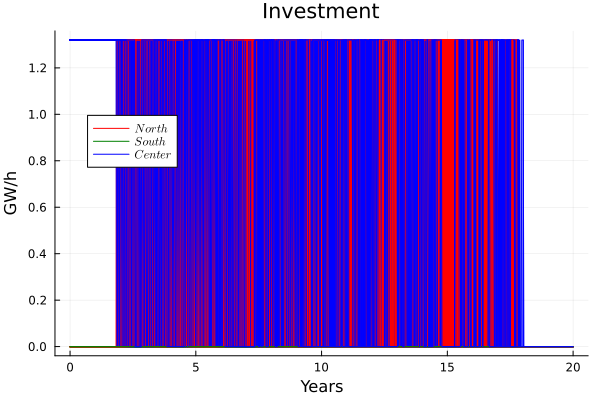}
        \caption{Long-term investment}
        \label{fig:slt-investment}
    \end{subfigure}
    \quad
    \begin{subfigure}[b]{0.33\linewidth}
        \centering
        \includegraphics[width=\linewidth]{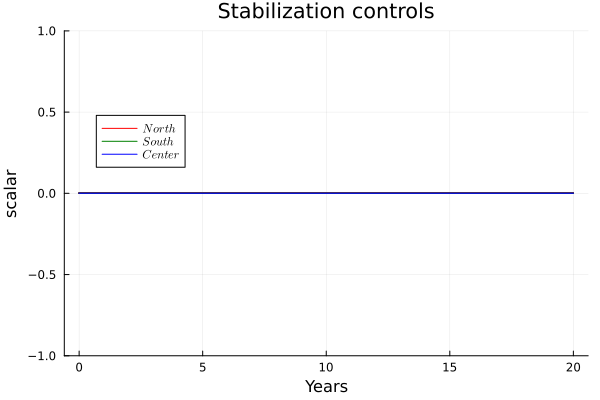}
        \caption{Long-term stabilization}
        \label{fig:slt-stabilization}
    \end{subfigure}
    \caption{Long-term control $(\mu,\alpha)$}
    \label{fig:slt-combined-nu}
\end{figure}

The controls $\nu = (\mu, \alpha)$ displayed in Figure \ref{fig:slt-combined-nu} appear similar to those in Figure \ref{fig:lt-combined-nu}, as no stabilization effort is applied in either case. However, there is a fundamental theoretical difference. In the mid-term setting, this is due to the large initial capacity. In the long-term setting, since $\hat Q^i < D_{20}^i$ at every node $i$, it follows that for any $U_{\bar\mu}$-valued control $\mu$, choosing $\nu \equiv (\mu,0) \in \Uc$\footnote{Here, $0$ denotes the random variable $\alpha : [0,T]\times \Omega \to \R^N$, defined by $(t,\omega) \mapsto (0,\cdots,0)$.} results in a positive probability of exiting the set $K$ at some time. Therefore, it is optimal for the ISO to invest significantly in new energy capacity in order to reduce the probability of incurring stabilization costs. Stabilization efforts are only used in critical scenarios, which occur with low probability and are not observed in our simulations.
Regarding the investment control, we observe an initial preference for the Center, associated with discontinuing the use of gas. Once this transition is complete, investment alternates between the North and Center. However, in contrast to the mid-term case, this investment is sustained over a larger portion of the time horizon. This behavior is explained by the faster growth in demand: higher demand increases the marginal value of energy, so the Hamiltonian is optimized by keeping the investment control active for a longer period.

The continuous solution of \eqref{prob:nu-subproblem} under the control $\nu = (\mu,\alpha)$ is shown in Figure \ref{fig:slt-combined-qphi}, and the breakdown of production by energy source is shown in Figure \ref{fig:slt-types-of-production}.
\begin{figure}[h]
    \centering
    \begin{subfigure}[b]{0.33\linewidth}
        \centering
        \includegraphics[width=\linewidth]{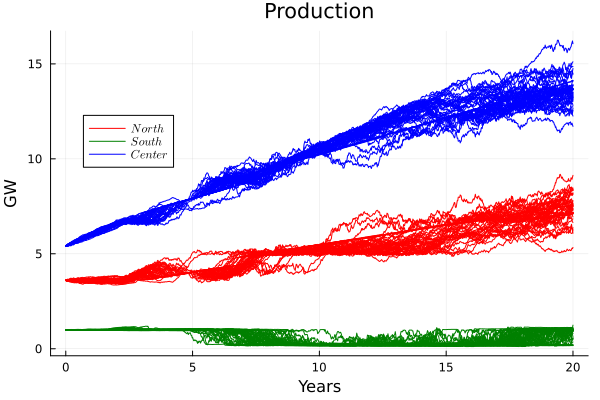}
        \caption{Long-term production}
        \label{fig:slt-production}
    \end{subfigure}
    \quad
    \begin{subfigure}[b]{0.33\linewidth}
        \centering
        \includegraphics[width=\linewidth]{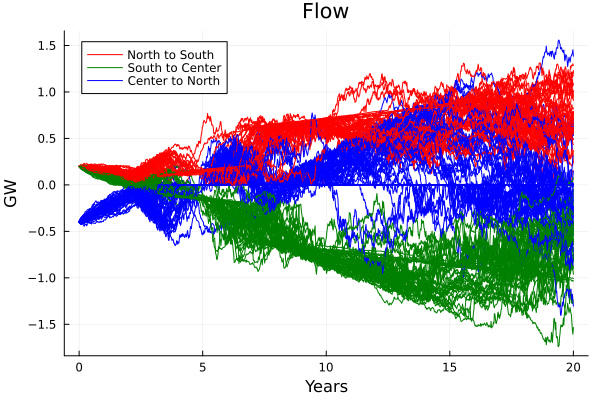}
        \caption{Long-term flow}
        \label{fig:slt-flow}
    \end{subfigure}
    \caption{Long-term control $(q,\phi)$}
    \label{fig:slt-combined-qphi}
\end{figure}
\begin{figure}[h!]
    \centering
    \begin{subfigure}[b]{0.30\linewidth}
        \centering
        \includegraphics[width=\linewidth]{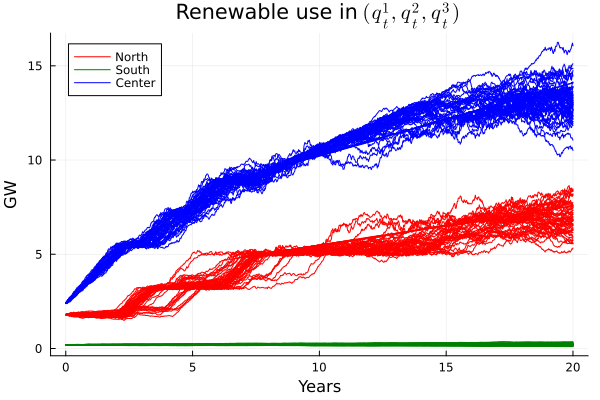}
        \caption{Renewable}
        \label{fig:slt-renewable}
    \end{subfigure}
    \hfill 
    \begin{subfigure}[b]{0.30\linewidth}
        \centering
        \includegraphics[width=\linewidth]{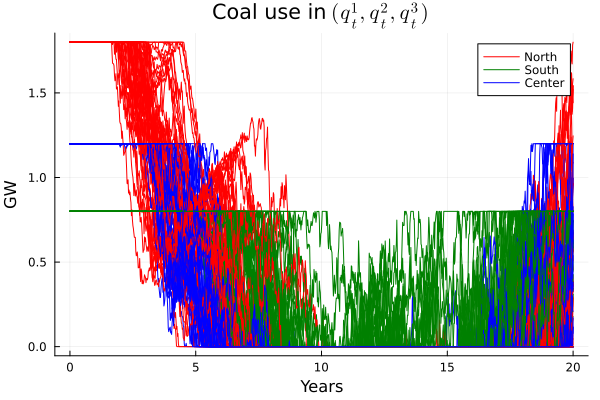}
        \caption{Coal}
        \label{fig:slt-coal}
    \end{subfigure}
    \hfill 
    \begin{subfigure}[b]{0.30\linewidth}
        \centering
        \includegraphics[width=\linewidth]{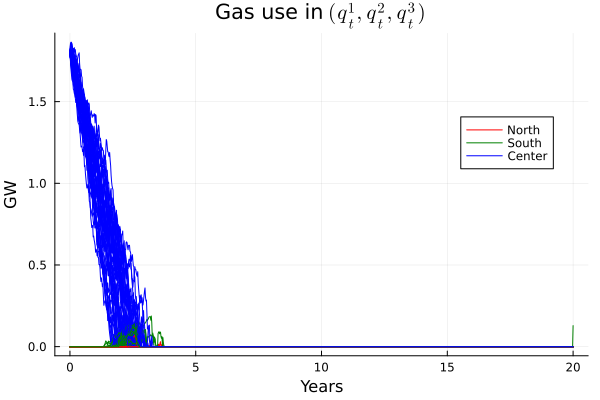}
        \caption{Gas}
        \label{fig:slt-gas}
    \end{subfigure}
    \caption{Long-term types of production}
    \label{fig:slt-types-of-production}
\end{figure}
Similarly to the previous setting, the production control evolves throughout the entire period. Its increase follows the same three main phases identified in the mid-term case.
The key difference between mid-term and long-term planning is that, in the latter, the increase in production must be sustained over a longer portion of the time horizon, as demand continues to grow and progressively approaches the initial available capacity.

\begin{remark}
    As the attentive reader may have already noticed, the finite difference scheme to solve the HJB equation that arises from Theorem \ref{Teo:HJB_CE} cannot handle the boundary conditions on $K$. Indeed, since the value function is only a supersolution on the boundary of the domain, its value on it is not determined. To the best of the authors’ knowledge, there is no standard framework to numerically solve state-constrained PDEs. We overcome this issue by using a penalization-type heuristic, in which we take a sufficiently large penalization parameter $\rho>0$, so that the HJB equation related to the unconstrained penalized problem will be given by
    \[ -\partial_t \varphi + H_{CE}(s,x,y,\partial_s\varphi, D_x \varphi, D_y \varphi, D^2_x\varphi) - \sum_{i=1}^N c^i(q^\star (x,y),x) = \rho d_K(x,y), \quad (t,s,x,y) \in [0,T) \times \R \times \R^N \times \R^N, \]
    where $d_K(x,y)$ denotes the distance between $(x,y)$ and the set $K$. We then expect the value functions of the penalized problems to converge monotonically to the constrained value function within the set $K$. The validity of this approach has been established in the deterministic setting by Bokanowski, Forcadel, and Zidani \cite{bokanowski2011deterministic}. Extending these results to the stochastic setting is beyond the scope of the present work.
\end{remark}

\section{Conclusions and future work} \label{sec:conclusion}

In this paper, we formulate and analyze the problem of an ISO seeking to minimize both operational and investment costs in an electricity network. Investment is directed toward renewable energy sources, whose intermittency introduces risk to the proper operation of the system. We provide a complete characterization of the problem within a stochastic control framework and establish well-posedness under realistic assumptions, namely Assumptions \ref{assump:IP-constant-A} and \ref{assump:C(Q,D) Lipschitz}. To the best of the authors’ knowledge, this is the first work to simultaneously account for both types of costs while incorporating a network structure that captures the geographical dimension of the electricity market. By first addressing the operational component, we are able to reformulate the investment problem as a stochastic control problem with state constraints.

\medskip
The numerical results in Section \ref{sec:Num-Results}, inspired in the Chilean network, highlight the relevance of several modeling features. In particular, the differentiation of energy sources already proves useful in the short-term planning setting, and remains important in the mid-term and long-term scenarios. However, the benefits of geographical differentiation and network structure become most apparent in the mid-term and long-term settings. The optimal investment strategy consistently avoids expanding capacity in the South, instead relying on transmission from the Center and North. This reflects the lower investment costs and greater potential for renewable development in those regions. These results suggest that incorporating a network structure is essential for accurately modeling large-scale, long-term investment in electricity systems.

\medskip
The state constraint introduces a significant technical challenge in the analysis. However, the numerical results, particularly the similarity between the solutions in Subsections \ref{subsec:lt-planning} and \ref{subsec:slt-planning}, indicate that optimal strategies are mostly unaffected by it. Differences arise only in rare events, or emergency situations, and in general it is optimal to avoid incurring stabilization costs. At the same time, Proposition \ref{prop:LPuniqueness} shows that the state constraint arises naturally from the operational component of the problem. This indicates that, even if it is not active most of the time, it remains a fundamental structural feature that should be preserved in more refined models of electricity markets. In this sense, the general framework developed in Section \ref{sec:State-Constr} is designed to accommodate future extensions. Overall, we believe the model represents a step forward in the continuous-time modeling of electricity markets with network structure.

\medskip
Regarding extensions of our work, a natural direction is to drop the assumption that the ISO’s interests are aligned with those of the producers. In that case, investment and operational decisions are no longer optimized jointly, but are instead determined through a sequential game in which producers choose investment first and anticipate the ISO’s operational response. Such formulation is closely related to the problem studied in this paper. Moreover, the arguments used in Proposition \ref{prop:LPuniqueness} suggest that it can be recast as a Nash-type differential game with a shared state constraint for the producers. This motivates the development of a theory of continuous-time stochastic games with state constraints on the controlled state process. To the best of the authors’ knowledge, a general treatment of this setting is still lacking, and it may have applications beyond electricity markets.

{\small
\bibliographystyle{plainurl} 
\bibliography{library}
}

\appendix

\section{Pseudo-Markov property} 

In this part, we use the setting presented in Section \ref{sec:State-Constr}. For $\omega \in \Omega$ and $ r \geq 0$, we denote $\omega^r_{\cdot} := \omega_{\cdot \land r}$ and $T_r(\omega)_{\cdot}:= \omega_{r \lor \cdot} - \omega_r$, from where $\omega = \omega^r + T_{r}(\omega)$. For $\nu \in \Uc(t,x)$, $\theta \in \mathcal{T}^{t}$, and $\tilde{\omega} \in \Omega$, we denote $\tilde{\nu}_\omega (\tilde{\omega}):= \nu(\omega^{\theta(\omega)} + T_{\theta(\omega)}(\tilde{\omega}))$. Clearly $\tilde{\nu}_\omega$ is an element of $\Uc(\theta(\omega), X^{\nu,t,x}_{\theta(\omega)})$ for each $\omega  \in \Omega$. 

\begin{lemma}\label{lemma:flow-property}
    Consider $f:\R^d \times U \to \R$ a measurable function, uniformly Lipschitz continuous in the first variable. For each $(t,x) \in [0,T) \times \R^d, \nu \in \Uc_t$ and $\theta \in \Tc^t$, we have that
    \begin{align*}
        \E \left[ \int_{t}^T f(X^{\nu,t,x}_s, \nu_s) \mathrm{d}s \bigg| \Fc_{\theta} \right] (\omega) = \int_t^{\theta(\omega)}f(X^{\nu,t,x}_s,\nu_s)(\omega)\mathrm{d}s +  J(\tilde{\nu}_\omega; \theta(\omega),X^{\nu,t,x}_{\theta(\omega)}(\omega)), \quad \P\text{-a.s.}
    \end{align*} 
\end{lemma}
    \begin{proof}
    The result is true in the context of bounded measurable terminal costs, as is the main result in \cite{claisse2016pseudo}, so the proof is focused on approximating our running cost by this type of function. First, given $(t,x,z)\in \R \times \R^d \times \R$, let us consider the system $(X,Z)^{t,x,z,\nu}$ where $X$ follows the SDE \eqref{eq:SDE estandar} and $Z$ follows
        \begin{align*}
            Z^{t,x,z,\nu}_s = z + \int_t^s f(X^{t,x,\nu}_s, \nu_s)\text{d}s.
        \end{align*}
        Let $\pi :\R^{d+1} \to \R$ be the projection of the last coordinate $\pi(x) = x^{d+1}$. It is clear that
        \begin{align*}
            \pi(X^{t,x,\nu}_T,Z^{t,x,0,\nu}_T) = Z^{t,x,0,\nu}_T =\int_{t}^T f(X^{t,x,\nu}_s,\nu_s)\text{d}s.
        \end{align*}
        Considering $\pi_n(x) = \pi(x) \cdot 1_{|\pi|\leq n}$ for each $ n \in \N$, clearly $|\pi_n(x,z)|\leq |\pi(x,z)|, \forall (x,z) \in \R^{d+1}$ and, similar to Proposition \ref{prop:PolynomialGrowth}, we know by the estimates for SDEs that 
        \begin{align*}
            \E\left[ |\pi(X^{t,x,\nu}_T ,Z^{t,x,0,\nu}_T ) |\right] = \E \left[ |Z^{t,x,0,\nu}_T| \right] \leq C(1+|x|^2)<\infty,
        \end{align*}
        with $C>0$. Then, from the dominated convergence theorem, we obtain that
        \begin{align}\label{limit:pseudo-markov}
            \E\left[ \pi_n(X^{t,x,\nu}_T ,Z^{t,x,0,\nu}_T) \right] \longrightarrow \E \left[ \pi(X^{t,x,\nu}_T , Z^{t,x,0,\nu}_T) \right].
        \end{align}
        Secondly, we use \cite[Theorem 2]{claisse2016pseudo} on $\pi_n$ and get
        \begin{align}\label{eq:pseudo-markov-bounded}
             \E \left[ \pi_n \left(X^{t,x,\nu}_T , Z^{t,x,0,\nu}_T \right) | \Fc_\theta \right](\omega) = \E \left[ \pi_n \left(X^{\theta(\omega),X^{t,x,\nu}_{\theta(\omega)},\tilde{\nu}_\omega }_T , Z_T^{\theta(\omega), X^{t,x,\nu}_{\theta(\omega)},Z^{t,x,0,\nu}_{\theta(\omega)}, \tilde{\nu}_\omega } \right) \right] =:\gamma_n(\omega) \quad \P\text{-a.s.},
        \end{align}
        where we assume that this equality holds for all $n \in \N$ in a set $\Omega' \subseteq \Omega$ with $\P(\Omega') = 1$\footnote{It is enough to take the intersection of all the almost sure sets in \eqref{eq:pseudo-markov-bounded}.}. A direct result from the construction of $\gamma_n$ is that, following the same argument as before, we obtain that 
        \begin{align*}
            \gamma_n (\omega) \longrightarrow \gamma(\omega) := \E \left[ \pi \left(X^{\theta(\omega),X^{t,x,\nu}_{\theta(\omega)},\tilde{\nu}_\omega }_T , Z^{\theta(\omega), X^{t,x,\nu}_{\theta(\omega)},Z^{t,x,0,\nu}_{\theta(\omega)}, \tilde{\nu}_\omega } \right) \right] \quad \P\text{-a.s.}
        \end{align*}
        Thirdly, we have that $\gamma$ is the conditional expectation of $\pi$ with respect to $\Fc_\theta$. Indeed, given that for $A \in \Fc_\theta$ and for each $n \in \N$ we have
        \begin{align*}
            \int_{A} \gamma_n(\omega) \text{d}\P(\omega) = \int_{A} \pi_n(X^{t,x,z}_T , Z^{t,x,0,\nu}_T) (\omega) \text{d}\P(\omega),
        \end{align*}
        it follows, by using again the dominated convergence theorem on both sides, that
        \begin{align*}
            \int_A \gamma (\omega)\text{d}\P(\omega) = \int_A \pi(X^{t,x,z}_T , Z^{t,x,0,\nu}_T) (\omega) \text{d}\P(\omega),
        \end{align*}
         which, from the definition of conditional expectation, implies that $\gamma (\omega) = \E \left[ \pi(X^{t,x,z}_T , Z^{t,x,0,\nu}_T) | \Fc_\theta \right](\omega)$ for every $\omega \in \Omega'$. This gives us the equality
         \begin{align*}
             \E \left[\pi(X^{t,x,z}_T , Z^{t,x,0,\nu}_T) | \Fc_\theta\right](\omega) =  \E \left[ \pi \left(X^{\theta(\omega),X^{t,x,\nu}_{\theta(\omega)},\tilde{\nu}_\omega }_T , Z_T^{\theta(\omega), X^{t,x,\nu}_{\theta(\omega)},Z^{t,x,0,\nu}_{\theta(\omega)}, \tilde{\nu}_\omega } \right)\right]  \quad \P\text{-a.s.}
         \end{align*}
         Finally, let us go back to the Lagrangian formulation of the problem, the equality is translated into 
         \begin{align*}
             \E \left[ \int_t^T f(X^{t,x,\nu}_s, \nu_s)\text{d}s  \bigg| \Fc_{\theta} \right](\omega) &= \E \left[ Z_{\theta(\omega)}^{t,x,0,\nu} (\omega)+\int_{\theta(\omega)}^{T} f(X^{\theta(\omega),X^{t,x,\nu}_{\theta(\omega)},\tilde{\nu}_\omega} _s , (\tilde{\nu}_\omega)_s) \text{d}s \right] \quad \\ 
             & = Z^{t,x,0,\nu}_{\theta(\omega)}(\omega) + J(\tilde{\nu}_\omega ; \theta(\omega), X^{t,x,\nu}_{\theta(\omega)}(\omega)) \\
             &= \int_{t}^{\theta(\omega)} f(X^{\nu,t,x}_s ,\nu_s)(\omega) \text{d}s + J(\tilde{\nu}_\omega ; \theta(\omega), X^{t,x,\nu}_{\theta(\omega)}(\omega))
         \end{align*}
         for each $\omega \in \Omega'$, obtaining the desired result.
    \end{proof}
\section{Comparison Principle}
The comparison result we consider is analogous to \cite[Theorem A.3.]{bouchard2012weak} with some adapted arguments between the Mayer and Lagrange formulation. As is standard in comparison results for HJB equations, we want to build a contradiction by using a priori estimates of the difference between a supersolution and a subsolution. To do so, it is necessary to have an estimation result for the difference of the Hamiltonians. Let us define
\begin{align*}
    H(x,p,X) := \sup_{u \in U} \left\{ -\Lc^{u}(x,p,X)- f(x,u) \right\},
\end{align*}
which, by Berge's maximum theorem, is continuous, see \cite[Theorem 17.31]{Charalambo2006}. For a pair of matrices $X,Y \in S^d$, we say that $X\leq Y$ if and only if the matrix $(Y-X)\in S^d$ is semipositive definite. 

\begin{lemma}\label{lemma:pre-comparison}
    There exists $\gamma >0$ such that
    \begin{align*}
        \liminf_{\eta\to 0^{+}} (H(y,q,Y^\eta) - H(x,p,X^{\eta}))& \\\leq \gamma (|x-y|(1+|q|&+n^2|x-y|)+ (1+|x|)|p-q|+ (1+|x|^2)|Q|)
    \end{align*}
    for $(x,y)\in C$ with $|x-y|\leq 1$ and for all $(p,q,Q)\in \R^d \times \R^d \times S^{2d}$, $(X^\eta,Y^\eta)_{\eta>0}\subseteq S^d \times S^d$, and 
    $n\geq1$ such that
    \begin{align*}
        \left(\begin{array}{cc}
            X^{\eta} & 0 \\
            0 & -Y^\eta
        \end{array} \right) \leq A_n + \eta A^2_n, ~\forall \eta >0 ,
    \end{align*}
    where
    \begin{align*}
        A_n := n^2\left( \begin{array}{cc}
            I_d & -I_d \\
            -I_d &  I_d
        \end{array} \right) + Q .
    \end{align*}
\end{lemma}

\begin{proof}
    Let us fix $u \in U$, we have that
    \begin{align*}
        &-\Lc^u(y,q,Y^\eta) - f(y,u) + \Lc^u(x,p,X^\eta) + f(x,u) \\
       & =(b(x,u)-b(y,u))^\top q + b(x,u)^\top(p-q) - (f(y,u)-f(x,u)) -\frac{1}{2} \left( Tr(\sigma\sigma^\top (y,u)Y^\eta -\sigma \sigma^\top (x,u)X^\eta)  \right) \\
       &= (b(x,u)-b(y,u))^\top q + b(x,u)^\top (p-q) - (f(y,u)-f(x,u))\\& \quad + \frac{1}{2}\sum_{i=1}^d \left(\left(\begin{array}{c}
            \sigma(x,u)  \\
             \sigma(y,u)
       \end{array}\right)^{\cdot,i} \right)^\top \left( \begin{array}{cc}
           X^\eta & 0 \\
          0  & -Y^\eta
       \end{array} \right) \left(\begin{array}{c}
            \sigma(x,u)  \\
             \sigma(y,u)
       \end{array}\right)^{\cdot,i}
    \end{align*}
    where $z^{\cdot,i}$ denotes the $i$-th column of the vector $z$. Thanks to the uniform Lipschitz continuity of $b$ and $f$, and to the linear growth that results from it, there exists $\gamma_1>0$ such that 
    \begin{align*}
        (b(x,u)-b(y,u))^{\top}q + b(x,u)^{\top}(p-q) - (f(y,u)-f(x,u)) \leq \gamma_1 (|x-y|(1+|q|) + (1+|x|)|p-q|).
    \end{align*}
    Next, we recall that
    \begin{align*}
        \left( \begin{array}{cc}
            X^\eta & 0 \\
            0 & -Y^\eta
        \end{array} \right) \leq A_n + \eta A_n^2 \implies z^{\top} \left( \begin{array}{cc}
            X^\eta & 0 \\
            0 & -Y^\eta
        \end{array} \right)z \leq z^{\top} (A_n + \eta A^2_n)z, ~\forall z\in \R^{2d}.
    \end{align*}
    This, along with the definition of $A_n$, the Lipschitz continuity of  $\sigma$, and its linear growth ensures the existence of a constant $\gamma_2>0$ such that 
    \begin{align*}
        \left(\left(\begin{array}{c}
            \sigma(x,u)  \\
             \sigma(y,u)
       \end{array}\right)^{\cdot,i} \right)^{\top} \left( \begin{array}{cc}
           X^\eta & 0 \\
          0  & -Y^\eta
       \end{array} \right) \left(\begin{array}{c}
            \sigma(x,u)  \\
             \sigma(y,u)
       \end{array}\right)^{\cdot,i} \leq& n^2|\sigma^{\cdot, i}(x,u)- \sigma^{\cdot,i}(y,u)|^2 
        +\gamma_2(1+|x|^2 + |y|^2)|Q|\\& + (1+|x|^2+ |y|^2)O(\eta).
    \end{align*}
    Taking $\gamma>0$ large enough so both inequalities hold simultaneously, recalling that $|x-y|\leq 1$ and $\sup A - \sup B \leq \sup A-B$, we obtain
    \begin{align*}
        H(x,q,Y^\eta) - H(x,p,X^\eta)\leq& \gamma(|x-y|(1+|q|+n^2|x-y|) + (1+|x|)|p-q| +(1+|x|^2)|Q|)  \\&+ (1+|x|^2 + |y|^2)O(\eta),
    \end{align*}
    we conclude by taking $\liminf_{\eta \to 0^+}$ on both sides.
\end{proof}

\begin{theorem}\label{Teo:Comparison}  Let $u_1$ be an lower semicontinuous supersolution on $C$ and let $u_2$ be a upper  semicontinuous subsolution on $C^\circ$ of \eqref{eq:ConstrainedHJB}. If $u_1$ and $u_2$ have polynomial growth on $C$ and if $u_2$ is of class $R(C^\circ)$, then
\begin{align*}
    u_1(T,x) \geq u_2(T,x), ~\forall x \in C \implies u_1(t,x)\geq u_2(t,x), ~\forall (t,x) \in [0,T]\times C
\end{align*}
\end{theorem}
\begin{proof}
    Without loss of generality, we will prove the result for the equation
    \begin{align}\label{eq:altered-HJB}
        \rho \varphi - \partial_t\varphi + \sup_{u\in U} \left\{ - \Lc^{u}(\cdot, D\varphi,D^2\varphi) - f(\cdot,u) \right\} = 0,
    \end{align}
    since $\varphi$ is subsolution of \eqref{eq:ConstrainedHJB} if and only if $e^{-\rho t}\varphi$ is a subsolution of \eqref{eq:altered-HJB}, and analogously with the supersolution property. 
    
    Let us assume that $u_1(T,\cdot)\geq u_2(T,\cdot)$ on $C$. Due to the polynomial growth of both subsolution and supersolution, we take $p\geq 1$ and $A> 0$ to be such that $u_2 (t,x) - u_1(t,x)<A(1+|x|^p)$ for each $(t,x) \in [0,T]\times C$. Let us assume for contradiction that $\sup(u_2 - u_1) > 0$; then is possible to find $\iota>0$ and $(t_0,x_0) \in [0,T]\times C$ such that
    \begin{align}\label{ineq:xi}
        \xi := (u_2 - u_1 - 2\phi)(t_0,x_0) = \max_{[0,T]\times C} (u_2-u_1-2\phi)>0
    \end{align}
    where
    \begin{align*}
        \phi(t,x):= \iota e^{-\kappa t} (1+|x|^{2p}).
    \end{align*}
    The existence of a maximizer is guaranteed due to the polynomial growth of the difference of both functions and $\kappa>0$ is large enough so that the function
    \begin{align*}
        m(t,x):= -\rho \phi(t,x)+ \partial_t\phi(t,x) + \gamma((1+|x|)|D\phi(t,x)| + (1+|x|^2)|D^2\phi (t,x)|)
    \end{align*}
    is nonpositive on $[0,T]\times\R^d$\footnote{When computing the derivatives of $\phi$ explicitly, the existence of $\kappa>0$ is not hard to check.}, and $\gamma$ is the one from Lemma \ref{lemma:pre-comparison}. Note that the assumption $u_1(T,\cdot) \geq u_2(T,\cdot)$ on $C$ implies that $(t_0,x_0)\in [0,T)\times C$. Now we separate in cases depending on whether $x_0$ is in the boundary or not.

    \textbf{Case 1:} $x \in \partial C$. For each $n\geq1$, there exists a point $(t^n, x^n, s^n, y^n) \in ([0,T]\times C)^2$ satisfying
    \begin{align*}
        \Phi^n(t^n,x^n,s^n,y^n) = \max_{[0,T]\times C} \Phi^n, \quad \Phi^n(t,x,s,y):= u_2(s,y) - u_1(t,x) - \Theta^n(t,x,s,y)
    \end{align*}
    and 
    \begin{align*}
        \Theta^n(t,x,s,y) := \frac{1}{2}n^2 (|t+\lambda(n^{-1}) - s|^2 + \varepsilon|x + l(n^{-1}) -y|^2+ |t-t_0|^2 + |x-x_0|^4) + \phi(t,x)+\phi(s,y),
    \end{align*}
    with $\varepsilon>0$, and $l,\lambda$ given for $x_0$ by Definition \ref{def:RO}\footnote{This thanks to the assumption that $u_2$ is of class $R(C^\circ)$}. It follows from the definition of $(t^n,x^n,s^n,y^n)$, the continuity of the function $\phi$, and the continuity of $u_2$ along the trajectory $\varepsilon \to (t+\lambda(\varepsilon),x+l(\varepsilon))$ that
    \begin{align*}
       \Phi(t^n,x^n,s^n,y^n) &\geq \Phi^n(t_0,x_0,t_0+\lambda(n^{-1}), x_0+l(n^{-1})) \\
        &= (u_2 - u_1 - 2\phi)(t_0,x_0) + o(1) \\
        & = \xi +o(1).
    \end{align*}
    Due to the polynomial growth of the difference between $u_2-u_1$, the definition of $\Theta^n$, and the previous inequality, we have that $(t^n,x^n,s^n,y^n)$ stays in a compact set, therefore we can assume that, up to a subsequence, $(t^n,x^n,s^n,y^n) \to(t^\infty,x^\infty,s^\infty,y^\infty)\in ([0,T]\times C)^2$. Now, using again the definition of $\Theta^n$ we obtain that $t^\infty = s^\infty$ and $x^\infty = y^\infty$. We then have
    \begin{align*}
        \xi &= (u_2 - u_1 - 2\phi)(t_0,x_0) \\
        &  = \max_{[0,T]\times C}(u_2 - u_1 - 2\phi) \\
        & \geq (u_2- u_1-2\phi)(t^\infty,x^\infty) - |t^\infty - t_0|^2 - |x^\infty-x_0|^4 \\
        & \quad -\limsup_{n \to \infty} \frac{1}{2} n^2 (|t^n + \lambda(n^{-1})-s^n|^2 + \varepsilon|x^n + l(n^{-1})-y^n|^2) \\
        & \geq \liminf_{n\to \infty} \Phi^n(t^n,x^n,s^n,y^n) \\
        & \geq \xi,
    \end{align*}
    which implies that $(t^\infty,x^\infty) = (t_0,x_0)$ and 
    \begin{align*}
        \limsup_{n\to\infty} \frac{1}{2}n^2(|t^n + \lambda(n^{-1})- s^n|^2 + \varepsilon|x^n + l(n^{-1})-y^n|^2) = 0.
    \end{align*}
    After passing to another subsequence, we deduce that
    \begin{align}
        \label{lim:A7} (t^n,x^ns^n,y^n) &\to (t_0,x_0,t_0,x_0) \\
        \label{lim:A8} u_2(s^n,y^n)- u_1(t^n,x^n) &\to (u_2-u_1)(t_0,x_0) \\
         \label{eq:A9} s^n = t^n+\lambda(n^{-1})+o(n^{-1}),& \quad y^n = x^n+l(n^{-1}) +o(n^{-1}).
    \end{align}
    Now, since $(t_0,x_0)\in [0,T)\times \partial C$, it follows from \eqref{in:defR(C)-2} and \eqref{eq:A9} that $(s^n,y^n)\in [0,T)\times C^\circ$ for $n$ large enough.

    Let $\overline{\Pc}^{2,-}_{C^{\circ}}u_1$ and $\overline{\Pc}^{2,+}_{C^\circ}u_2$ be the closed parabolic sub- and superjets as defined in \cite[Section 8]{crandall1992user}. From the Crandall-Ishii's Lemma for parabolic problems, \cite[Theorem 8.3.]{crandall1992user}, we obtain, for each $\eta >0$, elements 
    \begin{align*}
        (a^n,p^n,X^n_\eta) \in \overline{\Pc}^{2,+}_{C^\circ}u_2(s^n,y^n) \quad (b^n,q^n,Y^n_\eta) \in \overline{\Pc}^{2,-}_{C^{\circ}}u_1(t^n,x^n)
    \end{align*}
    such that
    \begin{align*}
        &a^n = \partial_t \Theta^n(t^n,x^n,s^n,y^n), \quad b^n = - \partial_s\Theta(t^n,x^n,s^n,y^n), \\
        &p^n = D_x\Theta(t^n,x^n,s^n,y^n), \quad q^n=  -D_y\Theta(t^n,x^n,s^n,y^n),
    \end{align*}
    and
    \begin{align*}
        \left( \begin{array}{cc}
            X^n_\eta &0  \\
            0 & -Y^n_\eta
        \end{array} \right) \leq A_n+\eta A^2_n,
    \end{align*}
    where $A_n = D^2\Theta(t^n,x^n,s^n,y^n)$; i.e.
    \begin{align*}
        A_n = \varepsilon n^2 \left( \begin{array}{cc}
            I_d & -I_d \\
            -I_d & I_d
        \end{array} \right) + \underbrace{\left( \begin{array}{cc}
            D^2\phi (t^n,x^n)+ O(|x^n-x_0|^2) & 0  \\
            0 & D^2\phi(s^n,y^n)
        \end{array} \right)}_{=:Q^n}
    \end{align*}
    In view of the super- and subsolution property of $u_1$ and $u_2$, the fact that $(s^n,y^n)\in [0,T)\times C^\circ$ for $n$ large, we obtain
    \begin{align*}
        \Delta_n &:= \rho(u_2(s^n,y^n) - u_1(t^n,x^n)) \\
            &\leq  (\partial_s\phi(s^n,y^n) +n^2(s^n-(t^n+\lambda(n^{-1})))  -H(x^n,p^n, X^n_\eta))  \\
            &\quad + (\partial_t\phi(t^n,x^n) +n^2(t^n+\lambda(n^{-1})-s^n))+2(t-t_0) + H(y^n,q^n,Y^n_\eta) ) \\
            &= \partial_s\phi(s^n,y^n) + \partial_t\phi(t^n,x^n) + 2(t-t_0) + (H(y^n,q^n,Y^n_\eta) - H(x^n,p^n,X^n_\eta)),
    \end{align*}
    by taking $\liminf_{\eta \to 0^+}$ in both sides and applying Lemma \ref{lemma:pre-comparison}, we have
    \begin{align*}
        \Delta_n &\leq \partial_s\phi(s^n,y^n) + \partial_t\phi(t^n,x^n) + 2(t-t_0) \\
        &\quad  +\gamma(|x^n-y^n|(1+|q^n|+n^2|x^n-y^n|) + (1+|x^n|)|p^n-q^n| + (1+|x^2|)|Q^n|).
    \end{align*}
    By the definition of $p^n$ and $q^n$, it follows that
    \begin{align*}
        \Delta_n &\leq 2(t-t_0) + \partial_s(s^n,y^n) + \partial_t(t^n,x^n) \\
        & \quad + \gamma|x^n-y^n|(1+|D \phi(s^n,y^n)| + \varepsilon n^2|x^n+l(n^{-1} )- y^n| + n^2|x^n-y^n|) \\
        & \quad + \gamma(1+|x^n|)(|D\phi(s^n,y^n)| + |D\phi(x^n,y^n)|+4|x^n-x_0|^3) \\
        & \quad + \gamma(1+|x^n|^2)(|D^2\phi(s^n,y^n)| + |D^2\phi(s^n,y^n)| + O(|x^n-x_0|^2))
    \end{align*}
    Recalling \eqref{lim:A7}-\eqref{eq:A9}, as $n\to \infty$, this inequality implies that
    \begin{align*}
        \rho(u_2 - u_1)(t_0,x_0) &\leq  2\partial_t\phi(t_0,x_0) + \gamma \varepsilon(\liminf_{n\to \infty} nl(n^{-1}))^2 \\
        &\quad +2\gamma ((1+|x_0|)|D\phi(t_0,x_0)| + (1+|x_0|^2)|D^2\phi(t_0,x_0)|).
    \end{align*}
    Thanks to \eqref{ineq:defR(C)-1}, we can take $\varepsilon\to 0$ and obtain that
    \begin{align*}
        0< \xi =\rho(u_2 - u_1-2\phi) (t_0,x_0)  \leq 2m(t_0,x_0).
    \end{align*}
    Since $\kappa$ was chosen with the purpose of $m$ being nonpositive on $[0,T]\times \R^d$, this contradicts \eqref{ineq:xi}.

    \textbf{Case 2}: $x \in C^\circ$. This case is simpler, since the boundary is not a problem we have to continuously worry about. This case is handled by using
    \begin{align*}
        \Theta^n(t,x,s,y):= \frac{1}{2}n^2 (|t-s|^2 + |x-y|^2) + |t-t_0|^2 + |x-x_0|^4  + \phi(t,x) + \phi(s,y),
    \end{align*}
    where with $n$ large enough, implies that $x^n,y^n\in C^{\circ}$. The rest of the proof use the same arguments and the Case 1.

\end{proof}

\end{document}